\documentclass[twoside]{article}%

\usepackage{mathtools}
\usepackage{latexsym}
\usepackage{mathrsfs}
\usepackage{listings}
\usepackage{epigraph}
\usepackage{hyperref}
\usepackage{pgfplots}
\usepackage{appendix}
\hypersetup{
  unicode=true,
  pdfauthor={},
  pdftitle={Redux - Ize Conjecture},
  pdfsubject={},
  pdfkeywords={}; 
}
\usepackage{pst-all}	
\usepackage{xcolor}
\usepackage{enumitem}
\usepackage{todonotes}
\usepackage{amssymb,bm}
\usepackage{amsmath}
\usepackage{amsfonts}
\usepackage{amsthm}
\usepackage{tikz}
\usepackage{epsfig}
\usepackage{verbatim}
\usepackage{float}
\usepackage{tabularx}
\usepackage{tikz-cd}

\providecommand{\U}[1]{\protect\rule{.1in}{.1in}}

\newcolumntype{Y}{>{\raggedleft\arraybackslash}X}
\def\bc{{\mathbb{C}}}

\def\bn{{\mathbb{N}}}

\def\bz{{\mathbb{Z}}}

\def\br{\mathbb R}

\def\vs{\vskip.3cm}
\def\noi{\noindent}
\def\wt{\widetilde}
\def\gdeg{G\text{\rm -deg}}

\def\Om{\Omega}

\def\ker{\text{\rm Ker\,}}

\DeclareMathOperator{\id}{Id}

\definecolor{codegray}{rgb}{0.5,0.5,0.5}
\newrgbcolor{Magenta}{1  0  1}
\definecolor{codegray}{rgb}{0.5,0.5,0.5}
\newrgbcolor{RoyalPurple}{.25  .10  1}
\newrgbcolor{NavyBlue}{.06  .46  1}
\definecolor{cadmiumgreen}{rgb}{0.0, 0.42, 0.24}
\newrgbcolor{Black}{1  1  1}

\lstdefinestyle{mystyle}{
    commentstyle=\color{codegreen},
    keywordstyle=\color{magenta},
    numberstyle=\tiny\color{codegray},
    stringstyle=\color{RoyalPurple},
    basicstyle=\ttfamily\footnotesize,
    breakatwhitespace=false,         
    breaklines=true,                 
    captionpos=b,                    
    keepspaces=true,                 
    numbers=left,                    
    numbersep=5pt,                  
    showspaces=false,                
    showstringspaces=false,
    showtabs=false,                  
    tabsize=2
}
  \lstdefinelanguage{GAP}{
    basicstyle=\ttfamily,
    keywords={true, false, function, return, fail, if, in, while, do, od, else, elif, fi, break, continue},
    keywordstyle=\color{NavyBlue}\bfseries,
    otherkeywords={
      >, <, ==
    },
    breaklines=true,      
    identifierstyle=\color{black},
    sensitive=True,
    comment=[l]{\#},
    commentstyle=\color{cadmiumgreen},
    stringstyle=\color{black},
    morestring=[b]',
    morestring=[b]"
  }
\newtheorem{theorem}{Theorem}[section]
\newtheorem{proposition}[theorem]{Proposition}
\newtheorem{lemma}[theorem]{Lemma}
\newtheorem{corollary}[theorem]{Corollary}

\newtheorem{definition}[theorem]{Definition}
\newtheorem{remark}[theorem]{Remark}
\newtheorem{example}[theorem]{Example}

\usepackage{subcaption}

\begin{document}

\title{The Ize Conjecture Redux: A Parity Criterion for Global Equivariant Bifurcation Guarantees}

\author{
Ziad Ghanem\thanks{\small Department of Mathematics, Brandeis University, Waltham, MA 02453, USA
}}
\date{}

\maketitle

\begin{abstract}
The Ize Conjecture proposed that every absolutely irreducible representation of a compact Lie group admits a maximal isotropy subgroup with an odd-dimensional fixed-point space, which would provide a universal bifurcation guarantee via the equivariant degree. Its disproof by Lauterbach and Matthews necessitates a more targeted criterion. We introduce Ize pairs---pairs $(G, V)$ for which some maximal isotropy subgroup $H$ satisfies $\dim V^H - \dim V^G \equiv 1 \pmod{2}$---and prove that this dimension-parity condition completely captures the algebraic obstruction to a non-trivial equivariant degree change at 
maximal orbit types. Integrating this criterion with a mod-2 equivariant spectral flow yields local and global bifurcation guarantees without recourse to Burnside ring arithmetic. As an application, we establish unbounded branches of non-stationary periodic solutions in a $\Gamma$-symmetric coupled oscillator network, where the bifurcation guarantees follow entirely from the crossing parity of the linearization at the boundary of a regular parameter window.
\end{abstract}

\noi \textbf{Mathematics Subject Classification:} Primary: 34C25, 37C81, 47H11, 55M25, 19A22

\medskip

\noi \textbf{Key Words and Phrases:}  symmetric equation, existence of solutions, equivariant
Leray-Schauder degree, nonlinear analysis, equivariant bifurcation.

\setlength{\epigraphwidth}{0.65\textwidth}
\epigraph{We must concede that the universe we see is a ceaseless creation, evolution, and destruction of forms and that the purpose of science is to foresee this change of form and, if possible, explain it.}{Ren\'e Thom, Structural Stability and Morphogenesis}

\section{Introduction} \label{sec:introduction}
Symmetry breaking is an organizing principle in nature: circular fluid layers heated from below buckle into convection rolls \cite{GSS1988, Rudge2024}; homogeneous elastic rods under compression snap into preferred bending planes \cite{GSS1988, Vanderbauwhede1980}; rings of identical oscillators spontaneously synchronize into traveling waves \cite{GolStew2002}. In each case, a highly symmetric equilibrium gives way to a less symmetric state, and the \textit{structure} of the underlying symmetry group dictates which new patterns can emerge.  Making this principle rigorous---and, ideally, algorithmic---has driven more than four decades of research in equivariant bifurcation theory \cite{survey, book-new, AED, ChossatLauter2000, GSS1988, GolStew2002}.
\vs

While the simplest manifestation of symmetry breaking is the loss of stability of a stationary equilibrium, the same group-theoretic mechanism governs the emergence of dynamic patterns such as periodic orbits and traveling waves. For expository clarity, we begin with the stationary case. Let $G$ be a compact Lie group acting orthogonally on a finite-dimensional state space $V$, and consider a parameterized system 
\begin{equation}\label{eq:intro_eq}
    \frac{dx}{dt} = f(x, \lambda),
\end{equation}
where $f\colon V \times \br \to V$ is a smooth $G$-equivariant vector field satisfying:
\begin{enumerate}
  \item[(T)] \textbf{Trivial branch:} $f(0,\lambda) = 0$ for all
    $\lambda \in \br$;
  \item[(C)] \textbf{Critical point:} there exists an isolated
    $\lambda_0 \in \br$ at which
    $Df_x(0,\lambda_0)\colon V \to V$ is singular.
\end{enumerate}
Stationary equilibria of \eqref{eq:intro_eq} are the solutions of the algebraic bifurcation problem
\begin{equation}\label{eq:intro_bif}
  f(x,\lambda) = 0.
\end{equation}
The Equivariant Branching Lemma (EBL) (Cicogna \cite{Cicogna1981}; Vanderbauwhede \cite{Vanderbauwhede1980}) illustrates the central role of fixed-point subspaces in this setting: if $G$ acts absolutely irreducibly on~$V$ and an isotropy subgroup $H \leq G$ satisfies $\dim V^H = 1$, then a unique branch of $H$-symmetric solutions to \eqref{eq:intro_bif} bifurcates from the trivial solution at $(\lambda_0,0)$, assuming the critical eigenvalue crosses zero transversally. The power of this result lies in its convenience: no detailed knowledge of the nonlinearity~$f$ is required beyond (T) and~(C); the only essential input is the \textit{geometry of the group action}, encoded in the dimension of the $H$-fixed-point space. However, the EBL has two significant limitations. First, it is a local result: the Implicit Function Theorem guarantees a solution branch only in an arbitrarily small neighborhood of the bifurcation point, with no control over the branch's global continuation. Second, the EBL detects only stationary bifurcations from equilibria; it does not extend to the emergence of periodic orbits or other dynamic patterns, which require functional-analytic tools.
\vs
A more general framework is provided by the equivariant degree theory. By replacing the Implicit Function Theorem with topological invariants---specifically, the equivariant Krasnosel'skii theorem and the Rabinowitz alternative---one can accommodate infinite-dimensional phase spaces and provide global continuation guarantees: if the equivariant degree changes across an isolated critical point, the bifurcating branch cannot spontaneously terminate in the interior of the phase space, but must either grow unbounded or reconnect to the trivial branch at another critical point. In this setting, a $G$-symmetric bifurcation problem is reformulated as a parametrized operator equation
\[
  \mathscr{F}(\alpha,u) = 0, \qquad
  \mathscr{F}\colon \br \times \mathscr{H} \to \mathscr{H},
\]
where $\mathscr{H}$ is an isometric Banach $G$-representation and $\mathscr{F}$ is a completely continuous $G$-equivariant field satisfying $\mathscr{F}(\alpha,0) = 0$ for all $\alpha \in \br$ (cf.\ Section~\ref{sec:bifurcation_setup}). Supposing that $\mathscr{F}$ admits a linearization $\mathscr{A}(\alpha) := D_u\mathscr{F}(\alpha,0)\colon \mathscr{H} \to \mathscr{H}$, and that $\alpha_0 \in \br$ is an isolated parameter value at which $\mathscr{A}(\alpha_0)$ is singular, flanked by regular values $\alpha_0^- < \alpha_0 < \alpha_0^+$, the \emph{local bifurcation invariant} at $(\alpha_0,0)  \in \br \times \mathscr H$ is the Burnside ring element
\begin{equation*}
  \omega_G(\alpha_0)
  := \gdeg(\mathscr{A}(\alpha_0^-), B(\mathscr{H}))
   - \gdeg(\mathscr{A}(\alpha_0^+), B(\mathscr{H}))
  \;\in A(G).
\end{equation*}
By the equivariant Krasnosel'skii theorem (Theorem~\ref{thm:abstract_local_bif}), whenever an orbit type $(H)$ carries a non-zero coefficient in $\omega_G(\alpha_0)$, a continuum of non-trivial solutions with symmetries at least $(H)$ emerges from the trivial branch at $(\alpha_0,0)$.
\vs
In practice, evaluating the local bifurcation invariant requires decomposing each degree \\ $\gdeg(\mathscr{A}(\alpha_0^\pm), B(\mathscr{H}))$ into a product of basic degrees---defined for each irreducible $G$-representation $\mathcal V$ as $\gdeg(-\id, B(\mathcal V)) \in A(G)$---and computing their difference within the multiplicative structure of the Burnside ring (for examples of this standard equivariant-degree pipeline for both existence and bifurcation problems, see \ \cite{BHKL, BBKX, BHKR, BalBurnett, BalChen, survey, book-new, AED, Duan, Eze2022, Carlos2019, Ghanem_annales, Ghanem2024, Krawcewicz2017}). For symmetry groups arising in physical applications, this computation involves recurrence formulae that must be evaluated sequentially across the isotropy lattice, a process whose complexity grows rapidly with the number of orbit types. The present paper circumvents this computation for maximal orbit types by exploiting a structural simplification: at the top of the isotropy lattice, the recurrence formula collapses, and the coefficient of a maximal orbit type $(H)$ in the equivariant degree depends only on the parity of the negative spectrum restricted to $H$-fixed-point subspaces (Lemma~\ref{lemm:nontriviality_conditions}). This collapse separates the spectral data of the operator---which varies with the parameter---from a fixed geometric property of the group action: the parity of $\dim \mathcal{V}^H$ for each irreducible subrepresentation $\mathcal{V}$. Thus, two data points govern the detection of maximal symmetry breaking at the top of the isotropy lattice: the spectral parity of the linearization, and the dimensional parity of maximal fixed-point spaces in the relevant irreducible representations. Whether the latter is automatically satisfied leads to the following natural question: \textit{does every absolutely irreducible representation admit a maximal isotropy subgroup with an odd-dimensional fixed-point space?}
\vs
This is precisely what the Ize Conjecture asserts. Posed by Ize (cf.\ Field~\cite{Field1996}), the conjecture proposed that every absolutely irreducible real representation of a compact Lie group admits at least one maximal isotropy subgroup with an odd-dimensional fixed-point space---guaranteeing, if true, that spectral parity alone would suffice for bifurcation guarantees. However, Lauterbach and Matthews~\cite{Lauter2010, Lauter2015} disproved the conjecture by constructing absolutely irreducible representations in $\br^4$ and $\br^8$ where every maximal isotropy subgroup has an even-dimensional fixed-point space. Although the disproof removes the prospect of a universal guarantee, the equivariant degree methodology remains highly effective because its success does not require the parity condition to hold for every conceivable representation, but only for the specific representations dictated by the physical system under investigation. By verifying this geometric condition for the relevant symmetry group, the degree-theoretic machinery proceeds unhindered. This observation motivates the central concept of this paper.
\begin{definition}\label{def:ize_pair}
  Let $G$ be a compact Lie group with the irreducible representations $\operatorname{Irr}(G) = \{ \mathcal V_i \}_{i \in \mathcal I}$, where $\mathcal V_0 \simeq \br$ denotes the trivial irreducible representation, and let $\mathbb{V}$ be an orthogonal $G$-representation with $G$-isotypic decomposition $\mathbb{V} \simeq \bigoplus_{i \in \mathcal{I}} m_i \mathcal{V}_i$. The pair $(G, \mathbb{V})$ is called an \textbf{Ize pair} if there exist a closed subgroup $H \leq G$ corresponding
  to a \textbf{maximal} orbit type in $\Phi_0(G;\mathbb{V})$ and an \textbf{odd} number of
  indices $i \in \mathcal{I} \setminus \{0\}$ satisfying
  \begin{enumerate}
    \item[(i)] $m_i \equiv 1 \pmod{2}$; \textup{AND}
    \item[(ii)] $\dim \mathcal{V}_i^H \equiv 1 \pmod{2}$.
  \end{enumerate}
\end{definition}
\begin{remark}\rm 
 In Section~\ref{sec:maximality}, we prove that these two conditions collapse into a single checkable criterion: the pair $(G, \mathbb{V})$ is an Ize pair if and only if $\dim \mathbb{V}^H - \dim \mathbb{V}^G \equiv 1 \pmod{2}$ for some maximal isotropy subgroup $H$ (Theorem~\ref{thm:parity_guarantee}).   
\end{remark}
To connect this algebraic criterion to the detection of bifurcation, we observe that the same parity structure governs the non-equivariant theory. In the classical setting, if an odd number of eigenvalues of the linearization cross the origin across a critical parameter value, the Leray--Schauder degree changes sign, and Krasnosel'skii's theorem guarantees bifurcation. The equivariant generalization requires tracking these crossings across individual $G$-isotypic components. The classical spectral flow, introduced by Atiyah, Patodi, and Singer~\cite{Atiyah1976} for self-adjoint Fredholm operators, was generalized by Izydorek, Janczewska, and Waterstraat~\cite{Izydorek2021} into a $G$-equivariant spectral flow taking values in the representation ring $RO(G)$ to detect periodic solutions in Hamiltonian systems. While powerful, their framework relies on finite-dimensional reduction and variational structures, limiting it to local bifurcation guarantees.
\vs
In this paper, we adapt the equivariant spectral flow to non-self-adjoint,  degree-theoretic settings. At each regular parameter value $\alpha$, the negative eigenspace $E^-(\alpha)$ of the linearization $\mathscr{A}(\alpha)$ is a finite-dimensional $G$-representation admitting a $G$-isotypic decomposition $E^-(\alpha) \simeq \bigoplus_{i \in \mathcal{I}} c_i(\alpha)
\mathcal{V}_i$, where $c_i(\alpha)$ counts the number of negative eigenvalues in the $i$-th $G$-isotypic component $\mathscr H_i \subset \mathscr H$ (modeled on the corresponding irreducible $G$-representation $\mathcal V_i$). The \emph{crossing parity} $\rho_i \in \{0,1\}$ records the net change (modulo~2) in these counts across an interval $[a,b]$ with regular endpoints and induces the \emph{spectral flow representation}
\[
  \mathbb{V}_{[a,b]}
  := \bigoplus_{i \in \mathcal{I}} \rho_i \mathcal{V}_i, \quad \rho_i \equiv c_i(a) - c_i(b) \pmod{2}.
\]
Our main result establishes that the Ize pair condition, when applied to the spectral flow representation $\mathbb{V}_{[a,b]}$ rather than to the ambient space $\mathscr{H}$, is sufficient for the non-vanishing of the local bifurcation invariant at maximal orbit types. More precisely, if $(G, \mathbb{V}_{[a,b]})$ constitutes an Ize pair with respect to a maximal orbit type $(H) \in \Phi_0(G; \mathbb{V}_{[a,b]})$, then $\operatorname{coeff}^H(\gdeg(\mathscr A(a), B(\mathscr H)) - \gdeg(\mathscr A(b), B(\mathscr H))) \neq 0$ (Proposition~\ref{prop:spectral_flow_evaluation}). Combined with Krasnosel'skii's theorem and the Rabinowitz alternative, this observation yields global, unbounded branches of non-trivial solutions with symmetries at least $(H)$ (see Theorems~\ref{thm:local_bif_ize} and \ref{thm:global_bif_ize}). The entire verification reduces to a single algebraic check---the parity of the difference $\dim \mathbb{V}_{[a,b]}^H - \dim \mathbb{V}_{[a,b]}^G$---evaluated from the spectral data at the boundaries of a regular parameter window.
\vs
\noindent\textbf{Organization.}
In Section~\ref{sec:maximality}, we establish the Parity Criterion characterization of Ize pairs (Theorem~\ref{thm:parity_guarantee}). In Section~\ref{sec:bifurcation_setup}, we embed this criterion into the parameterized equivariant degree framework, establishing rigorous local (Theorem~\ref{thm:local_bif_ize}) and global (Theorem~\ref{thm:global_bif_ize}) bifurcation guarantees through the mod-2 crossing parity of the linearization. In Section~\ref{sec:bifurcation_application}, we apply this machinery to a symmetric network of coupled oscillators, establishing the existence of unbounded branches of non-stationary periodic solutions, and provide a step-by-step worked-out example explicitly constructing the compliant index set for the case of $\Gamma = D_8$. Appendix~\ref{sec:product_geometry} details the fixed-point geometry in product groups required for these explicit computations. Appendix~\ref{sec:appendix} provides the axiomatic framework of the $G$-equivariant Leray--Schauder degree, and Appendix~\ref{sec:appendix_esf} develops the mod-2 $G$-equivariant spectral flow.

\section{A Complete Characterization of Ize Pairs} \label{sec:maximality}
The definition of an Ize pair (Definition~\ref{def:ize_pair}) involves two simultaneous conditions on the isotypic data.  In this section, we show that both conditions collapse into a single, checkable criterion: the parity of the difference $\dim \mathbb V^H - \dim \mathbb V^G$.
\vs
Let $G$ be a compact Lie group and assume that we have access to a complete list of the irreducible $G$-representations $\operatorname{Irr}(G) = \{\mathcal V_i\}_{i \in \mathcal I}$, where $\mathcal V_0 \simeq \br$ indicates the irreducible representation on which $G$ acts trivially. Let $\mathbb V$ be an orthogonal $G$-representation with the $G$-isotypic decomposition
\[
\mathbb V \simeq \bigoplus_{i \in \mathcal I} m_i \mathcal V_i.
\]
\begin{theorem}[Parity Criterion]\label{thm:parity_guarantee}
An isotropy subgroup $H \leq G$ satisfies the parity condition
\begin{equation*}
    \dim \mathbb V^H - \dim \mathbb V^G \equiv 1 \pmod{2},
\end{equation*}
if and only if there exist an \textbf{odd} number of non-trivial $G$-isotypic indices $i \in \mathcal I \setminus \{0\}$ satisfying
\begin{enumerate}\item[(i)] $m_i \equiv 1 \pmod{2}$ AND
\item[(ii)] $\dim \mathcal V_i^H \equiv 1 \pmod{2}$.
\end{enumerate}
\end{theorem}
\begin{proof}
Since $\mathcal V_0$ is the trivial representation, $\dim \mathcal V_0^H = \dim \mathcal V_0^G = 1$, and for any non-trivial irreducible $G$-representation with $i \neq 0$, one has $\dim \mathcal V_i^G = 0$. The $H$-fixed-point space decomposes isotypically as
$\mathbb V^{H} \simeq \bigoplus_{i \in \mathcal I}
(\mathcal{V}_i^{H})^{\oplus m_i}$, so the difference in fixed-point dimensions is given by
\[
  \dim \mathbb V^{H} - \dim \mathbb V^{G} = \sum_{i \in \mathcal I}
    m_i \cdot \dim \mathcal{V}_i^{H} - \sum_{i \in \mathcal I}
    m_i \cdot \dim \mathcal{V}_i^{G} = \sum_{i \in \mathcal I \setminus \{0\}}
    m_i \cdot \dim \mathcal{V}_i^{H}.
\]
A sum of non-negative integers is odd if and only if an odd number of its terms are odd, and a product of two integers is odd if and only if both factors are odd. Therefore the difference $\dim \mathbb V^H - \dim \mathbb V^G$ is odd if and only if an odd number of indices $i \in \mathcal I \setminus \{0\}$ satisfy both $m_i \equiv 1$ and $\dim \mathcal V_i^H \equiv 1 \pmod 2$, which is precisely the conjunction of~(i) and~(ii).
\end{proof}

\begin{corollary}[Characterization of Ize Pairs]
A pair $(G, \mathbb V)$ is an \textbf{Ize pair} if and only if there exists a subgroup $H \leq G$ corresponding to a \textbf{maximal} orbit type in the isotropy lattice $\Phi_0(G; \mathbb V)$, such that the difference $\dim \mathbb V^H - \dim \mathbb V^G$ is \textbf{odd}.
\end{corollary}

The Parity Criterion reduces the verification of the Ize pair property from an analysis of individual isotypic components to a single dimension count.  In the next section, we show how this criterion interfaces with the spectral data of a parametrized $G$-equivariant operator to produce both local and global bifurcation guarantees.

\section{Equivariant Bifurcation For Ize Pairs} \label{sec:bifurcation_setup}
The Parity Criterion of Section \ref{sec:maximality} is a statement about representations; it says nothing about the parameterized operators whose spectral changes drive bifurcation. To bridge this gap, we now embed the criterion into the degree-theoretic framework, showing that the spectral data at regular parameter values naturally assembles into a $G$-representation to which the Parity Criterion can be applied.
\vs
Let $G$ be a compact Lie group with irreducible representations $\operatorname{Irr}(G) = \{\mathcal V_i\}_{i \in \mathcal I}$, and let $\mathscr H$ be an isometric Banach $G$-representation admitting the $G$-isotypic decomposition
\[
  \mathscr H = \overline{\bigoplus_{i \in \mathcal I} \mathscr H_i},
  \quad \mathscr H_i \simeq m_i \mathcal V_i.
\]
where each $G$-isotypic component $\mathscr H_i$ is modeled on some number $m_i \in \bn \cup \{0\}$ of copies---its $G$-isotypic multiplicity---of the corresponding irreducible $G$-representation $\mathcal V_i$.
\vs
Let $\mathscr F: \br \times \mathscr H \to \mathscr H$ be a $G$-equivariant, completely continuous field satisfying $\mathscr F(\alpha,0) = 0$ for all $\alpha \in \br$ and consider the one-parameter bifurcation problem
\begin{equation}\label{eq:bifurcation_problem}
    \mathscr F(\alpha,u) = 0, \quad (\alpha,u) \in \br \times \mathscr H.
\end{equation}
Suppose that $\mathscr F$ admits a Fr\'echet derivative $\mathscr A(\alpha):=D_u \mathscr F(\alpha,0) : \mathscr H \to \mathscr H$ depending continuously on the bifurcation parameter $\alpha \in \br$. By Schur's Lemma, this linearization respects the $G$-isotypic decomposition
\[
\mathscr A(\alpha) = \bigoplus_{i \in \mathcal I} \mathscr A_i(\alpha), \quad \mathscr A_i(\alpha) := \mathscr A(\alpha)|_{\mathscr H_i}: \mathscr H_i \to \mathscr H_i.
\]
We write $M := \{(\alpha,0) : \alpha \in \br\}$ for the
\emph{trivial solution set} and $ \mathscr S := \{(\alpha,u) \in \br \times \mathscr H \setminus\{0\} :
       \mathscr F(\alpha,u) = 0\}$ 
for the set of \emph{non-trivial solutions}.  For any closed subgroup $H \leq G$, we denote by $\mathscr S^H := \{(\alpha,u) \in \mathscr S : G_u \geq H\}$
the subset of non-trivial solutions with \emph{symmetries at least~$(H)$}.
\subsection{The Local Bifurcation Invariant and Krasnosel'skii's Theorem}\label{sec:local-bif-inv}
We recall the standard terminology of parametrized bifurcation theory (cf.~\cite{book-new, AED}).  A trivial solution $(\alpha_0,0) \in M$ is a \emph{bifurcation point} for \eqref{eq:bifurcation_problem} if every neighborhood of $(\alpha_0,0)$ has nonempty intersection with~$\mathscr S$.  It is a \emph{regular point} if $\mathscr A(\alpha_0)$ is an isomorphism, and a \emph{critical point} otherwise.  A critical point $(\alpha_0,0)$ is \emph{isolated} if there exists $\varepsilon > 0$ such that $\mathscr A(\alpha)$ is an isomorphism for every
$\alpha \in (\alpha_0 - \varepsilon, \alpha_0 + \varepsilon) \setminus \{\alpha_0\}$.  The \emph{critical set} of \eqref{eq:bifurcation_problem} is
\begin{equation*}
  \Lambda := \{(\alpha,0) \in M :
      \dim \ker \mathscr A(\alpha) > 0\}.
\end{equation*}
A trivial solution $(\alpha_0,0) \in M$ is a \emph{branching point} if there exists a nontrivial continuum $K \subset \overline{\mathscr S}$ with
$K \cap M = \{(\alpha_0,0)\}$.  Any maximal connected subset $\mathscr C \subset \overline{\mathscr S}$ containing $(\alpha_0,0)$ is called a \emph{branch} bifurcating from~$(\alpha_0,0)$, and $\mathscr C$ is said to admit \emph{symmetries at least~$(H)$} if
$\mathscr C \cap \overline{\mathscr S^H} \neq \emptyset$.
\vs
\noindent\textbf{The Local Bifurcation Invariant.}
Let $(\alpha_0,0) \in \Lambda$ be an isolated critical point.  Choose $\varepsilon > 0$ such that $\mathscr A(\alpha)$ is an isomorphism for all $\alpha \in (\alpha_0 - \varepsilon, \alpha_0 + \varepsilon) \setminus \{\alpha_0\}$, and pick regular values
\[
  \alpha_0^\pm
  \in (\alpha_0 - \varepsilon, \alpha_0 + \varepsilon), \qquad
  \alpha_0^- < \alpha_0 < \alpha_0^+.
\]
Since $\mathscr A(\alpha_0^\pm)$ are nonsingular, there exists $\delta > 0$ such that the maps
$\mathscr F_\pm(u) := \mathscr F(\alpha_0^\pm, u)$ satisfy $\mathscr F_\pm^{-1}(0) \cap \partial B_\delta = \emptyset$. Moreover, one can demonstrate that $\mathscr F_\pm(u)$ are $B_\delta$-admissibly $G$-homotopic to $\mathscr A(\alpha_0^\pm)$. By homotopy invariance of the $G$-equivariant degree, one has
\[
  \gdeg(\mathscr F_\pm, B_\delta)
  = \gdeg(\mathscr A(\alpha_0^\pm), B(\mathscr H)),
\]
where $B(\mathscr H)$ is the open unit ball in~$\mathscr H$. The \emph{local bifurcation invariant} at $(\alpha_0,0)$ is the Burnside ring element
\begin{equation*}
  \omega_G(\alpha_0)
  := \gdeg(\mathscr A(\alpha_0^-), B(\mathscr H))
   - \gdeg(\mathscr A(\alpha_0^+), B(\mathscr H)).
\end{equation*}
It is standard (cf.~\cite{book-new, AED}) that
$\omega_G(\alpha_0)$ is independent of the choices of
$\alpha_0^\pm$ and~$\delta$ and that its non-triviality in $A(G)$ implies the emergence of a continuum of non-trivial solutions from the trivial branch at $(\alpha_0,0) \in \Lambda$:
\begin{theorem}[Krasnosel'skii-Type Local Bifurcation Result]\label{thm:abstract_local_bif}
Let $(\alpha_0,0)$ be an isolated critical point. If $\omega_{G}(\alpha_0)\neq 0$ in $A(G)$, then there exists a branch $\mathscr C\subset \overline{\mathscr S}$ of nontrivial solutions to \eqref{eq:bifurcation_problem} bifurcating from $(\alpha_0,0)$. Moreover, if an orbit type $(H)\in\Phi(G)$ satisfies 
    \[
\operatorname{coeff}^{H}\big(\omega_{G}(\alpha_0)\big)\neq 0,
\]
then $\mathscr C$ consists only of solutions with symmetries at least $(H)$, i.e.\ $\mathscr C\cap\overline{\mathscr S^H}\neq\emptyset$.
\end{theorem}

\subsubsection{Spectral Flow Evaluation of the Difference of two $G$-equivariant Degrees}
\label{sec:spectral_flow_eval}
The local bifurcation invariant $\omega_G(\alpha_0)$ is a difference of two $G$-equivariant degrees.  We now show that, for maximal orbit types, the coefficient of this difference can be read off directly from a finite-dimensional $G$-representation encoding the net spectral
change.
\vs
Let $a, b \in \br$ be any two regular parameter values of the linearization $\mathscr A(\alpha)$, and consider the difference of $G$-equivariant degrees 
\begin{align*}
    \gdeg(\mathscr A(a), B(\mathscr H)) - \gdeg(\mathscr A(b), B(\mathscr H)) \in A(G).
\end{align*}
At each regular parameter value $\alpha \in \mathbb{R}$, the space $\mathscr{H}$ admits a $G$-invariant splitting $\mathscr{H} = E^+(\alpha) \oplus E^-(\alpha)$, {\red where the negative eigenspace $E^-(\alpha)$ is defined as the direct sum of the generalized eigenspaces corresponding strictly to the real negative eigenvalues of $\mathscr{A}(\alpha)$ (complex eigenvalues of a real linear operator occur in conjugate pairs, so they do not impact the parity of the Brouwer degree and are naturally absorbed into $E^+(\alpha)$)}. Since the negative eigenspace $E^-(\alpha)$ is a finite-dimensional $G$-representation, it admits a $G$-isotypic decomposition $E^-(\alpha) \simeq \bigoplus_{i \in \mathcal I} c_{i}(\alpha) \mathcal V_{i}$, where $c_{i}(\alpha)$ counts the number of negative eigenvalues in the $i$-th $G$-isotypic component. We encode the net topological change across the interval $[a,b]$ into the \emph{spectral flow $G$-representation} 
\begin{align}\label{def:spectral_flow_representation}
 \mathbb V_{[a,b]} := \bigoplus_{i \in \mathcal I} \rho_i \mathcal V_i, \quad \text{where} \quad \rho_i \equiv c_i(a) - c_i(b) \pmod 2, \quad \rho_i \in \{0,1\}.   
\end{align}
\begin{proposition}[Spectral Flow Evaluation]\label{prop:spectral_flow_evaluation}
If $(G, \mathbb V_{[a,b]})$ is an \textbf{Ize pair} with respect to a maximal orbit type $(H) \in \Phi_0(G; \mathbb V_{[a,b]})$, then one has
\[
  \operatorname{coeff}^H \left(
    \gdeg(\mathscr A(a), B(\mathscr H))
    - \gdeg(\mathscr A(b), B(\mathscr H))
  \right) \neq 0.
\]
\end{proposition}
\begin{proof}
Put $\omega([a,b]) := \gdeg(\mathscr A(a), B(\mathscr H)) - \gdeg(\mathscr A(b), B(\mathscr H)) \in A(G)$ and, for each closed subgroup $K \leq G$, denote its local Brouwer degree difference by $\omega^K([a,b]) := \deg(\mathscr A(a)^K, B(\mathscr H^K)) - \deg(\mathscr A(b)^K, B(\mathscr H^K)) \in \bz$. 
By linearity of the coefficient operator, the Recurrence Formula (Appendix~\ref{sec:appendix}) gives
\begin{align*}
\operatorname{coeff}^H ( \omega([a,b]) )  \quad = \frac{\omega^H([a,b]) - \sum_{L > H} \operatorname{coeff}^L ( \omega([a,b]) ) |W(L)| n(H,L)}{|W(H)|}.
\end{align*}
For any subgroup $L \leq G$, let $\mu_L(\alpha) := \sum_{i \in \mathcal I} c_{i}(\alpha) \dim \mathcal V_{i}^L$ denote the dimension of the $L$-fixed-point space of the negative eigenspace at a regular parameter $\alpha$.
Notice that the local degree differences $\omega^K([a,b]) = (-1)^{\mu_K(a)} - (-1)^{\mu_K(b)}$ strictly take values in $\{-2, 0, 2\}$. As such, the recurrence relation operates entirely over even integers. Evaluating the sequence modulo 4 allows us to distinguish between terms that share the same parity of fixed-point dimensions (which vanish) and those with opposite parities (which evaluate to $2$), i.e.
\[
    \omega^K([a,b]) \equiv \begin{cases} 
        0 \pmod 4 & \text{if } \mu_K(a) - \mu_K(b) \equiv 0 \pmod 2, \\
        2 \pmod 4 & \text{if } \mu_K(a) - \mu_K(b) \equiv 1 \pmod 2.
    \end{cases}
\]
Since $(H)$ is a maximal orbit type in $\mathbb V_{[a,b]}$, the $L$-fixed-point space satisfies $\mathbb V_{[a,b]}^L = \mathbb V_{[a,b]}^G$ for all $(L) > (H)$, implying $(\mu_L(a) - \mu_L(b)) \equiv (\mu_G(a) - \mu_G(b)) \pmod{2}$. Under our modulo 4 encoding, this maximality guarantees
\begin{equation}\label{eq:parity_lock_1}
    \omega^L([a,b]) \equiv \omega^G([a,b]) \pmod 4 \quad \text{for all } (L) > (H).
\end{equation}
Conversely, since $(G, \mathbb V_{[a,b]})$ is an Ize pair, the Parity Criterion (Theorem~\ref{thm:parity_guarantee}) implies $\dim \mathbb V_{[a,b]}^H \not\equiv \dim \mathbb V_{[a,b]}^G \pmod 2$. Consequently, $\mu_H(a) - \mu_H(b)$ and $\mu_G(a) - \mu_G(b)$ must have strictly opposite parities, i.e.
\begin{equation}\label{eq:parity_lock_2}
    \omega^H([a,b]) \not\equiv \omega^G([a,b]) \pmod 4.
\end{equation}
We now evaluate the integers $\operatorname{coeff}^L ( \omega([a,b]) ) |W(L)| \pmod 4$ for all $(L) > (H)$ via downward induction. For the top element $G$, the recurrence sum is empty, giving $\operatorname{coeff}^G ( \omega([a,b]) ) |W(G)| = \omega^G([a,b])$. For any intermediate orbit type $H < L < G$, assume inductively that one has $\operatorname{coeff}^{\wt K} ( \omega([a,b]) ) |W(\tilde K)| \equiv 0 \pmod 4$ for all $L < \tilde K < G$. Evaluating the recurrence relation at $L$ modulo 4 yields
\begin{align*}
 \operatorname{coeff}^L ( \omega([a,b]) )|W(L)|  \equiv \omega^L([a,b]) - \sum_{L < \tilde K} \operatorname{coeff}^{\wt K} ( \omega([a,b]) ) |W(\tilde K)| n(L, \tilde K) \pmod 4.    
\end{align*}
Since $n(L,G) = 1$ for any subgroup $L$, and the intermediate terms vanish modulo 4 by the inductive hypothesis, this collapses to $\operatorname{coeff}^L ( \omega([a,b]) ) |W(L)| \equiv \omega^L([a,b]) - \omega^G([a,b]) \pmod 4$. By \eqref{eq:parity_lock_1}, this difference is exactly $0 \pmod 4$. Thus, $\operatorname{coeff}^L ( \omega([a,b]) ) |W(L)| \equiv 0 \pmod 4$ for all intermediate subgroups $H < L < G$.
Finally, we evaluate the relation at $H$:
\begin{align*}
\operatorname{coeff}^H ( \omega([a,b]) ) |W(H)| \equiv \omega^H([a,b]) - \sum_{H < L} \operatorname{coeff}^L ( \omega([a,b]) ) |W(L)| n(H,L) \pmod 4.    
\end{align*}
Since all intermediate terms $\operatorname{coeff}^L ( \omega([a,b]) ) |W(L)|$ vanish modulo 4, and $n(H,G) = 1$, we obtain $\operatorname{coeff}^H ( \omega([a,b]) ) |W(H)| \equiv \omega^H([a,b]) - \omega^G([a,b]) \pmod 4$. Since both values strictly belong to $\{0, 2\} \pmod 4$ and are unequal by \eqref{eq:parity_lock_2}, their difference must be exactly $2 \pmod 4$. Since $\operatorname{coeff}^H ( \omega([a,b]) ) |W(H)| \equiv 2 \pmod 4$, it follows that $\operatorname{coeff}^H ( \omega([a,b]) ) |W(H)| \neq 0$. In this way we have proven that the coefficient $\operatorname{coeff}^H(\omega([a,b]))$ must be non-zero.
\end{proof}
\subsection{Local Bifurcation for Ize Pairs}
\label{sec:local_bif_ize}
Proposition~\ref{prop:spectral_flow_evaluation} evaluates the degree difference between any two regular parameter values.  Specializing to the endpoints $\alpha_0^\pm \in \br$ flanking an isolated critical point $(\alpha_0,0) \in \Lambda$ yields our main local bifurcation result.
\begin{theorem}[Local Bifurcation for Ize Pairs]\label{thm:local_bif_ize}
Let $(\alpha_0,0) \in \Lambda$ be an isolated critical point of $\mathscr A(\alpha)$. If the spectral flow representation $(G, \mathbb V_{[\alpha_0^-,\alpha_0^+]})$ constitutes an \textbf{Ize pair} for a maximal orbit type $(H) \in \Phi_0(G; \mathbb V_{[\alpha_0^-,\alpha_0^+]})$, then $\operatorname{coeff}^H(\omega_G(\alpha_0)) \neq 0$ and there exists a branch $\mathscr C_H$ of \textbf{non-trivial} solutions to \eqref{eq:bifurcation_problem} bifurcating from $(\alpha_0,0)$ with symmetries at least $(H)$.
\end{theorem}
\begin{proof}
By definition, the local bifurcation invariant $\omega_G(\alpha_0)$ is the difference in the degrees at the regular endpoints $\alpha_0^-$ and $\alpha_0^+$. Setting $a = \alpha_0^-$ and $b = \alpha_0^+$, the spectral flow representation $\mathbb V_{[a,b]}$ from Proposition~\ref{prop:spectral_flow_evaluation} is exactly $\mathbb V_{[\alpha_0^-,\alpha_0^+]}$. Since $(G, \mathbb V_{[\alpha_0^-,\alpha_0^+]})$ is an Ize pair for maximal $(H)$, Proposition~\ref{prop:spectral_flow_evaluation} guarantees that $\operatorname{coeff}^H(\omega_G(\alpha_0)) \neq 0$. By Krasnosel'skii's Theorem (Theorem~\ref{thm:abstract_local_bif}), a branch $\mathscr C_H$ of non-trivial solutions with symmetries at least $(H)$ must bifurcate from the trivial branch at $(\alpha_0, 0)$.
\end{proof}
\begin{remark}\rm
The Ize pair hypothesis is imposed on the spectral flow representation $\mathbb V_{[\alpha_0^-,\alpha_0^+]}$, not on the ambient
space~$\mathscr H$.  By the Parity Criterion
(Theorem~\ref{thm:parity_guarantee}), verification reduces to checking whether some maximal isotropy subgroup $H$ satisfies $\dim\mathbb V_{[\alpha_0^-,\alpha_0^+]}^H - \dim\mathbb V_{[\alpha_0^-,\alpha_0^+]}^G \equiv 1 \pmod{2}$---a
single dimension count on the crossing data $c_i(\alpha_0^\pm)$, with no dependence on the multiplicative structure of the Burnside ring.
\end{remark}

\subsection{Global Bifurcation for Ize Pairs}\label{sec:global_bif_ize}
Theorem \ref{thm:local_bif_ize} guarantees the emergence of an $H$-symmetric branch at an isolated critical point, but provides no information about the branch's fate away from $(\alpha_0, 0)$: it could, a priori, bend back and reconnect to the trivial branch at another critical point, forming a closed loop with no large-amplitude solutions. To rule out this possibility and establish unboundedness, we invoke the Rabinowitz alternative and show that the Ize pair condition, when imposed on the global spectral flow across the entire critical set, forces the branch to escape any bounded region.
\vs
The following result is standard (cf.~\cite{AED, book-new}):
\begin{theorem}[The Rabinowitz Alternative]\label{thm:Rab}
Assume the critical set $\Lambda$ is finite:
\begin{enumerate}[label=($B$)]
\item\label{bB} $\Lambda = \{(\alpha_0,0),(\alpha_1,0),\ldots,(\alpha_n,0)\}$ with $\alpha_0 < \alpha_1 < \cdots < \alpha_n$.
\end{enumerate}
Let $\mathcal U \subset \br \times \mathscr H$ be an open bounded $G$-invariant set with $\partial \mathcal U \cap \Lambda = \emptyset$. If $\mathscr C$ is a branch of nontrivial solutions to $\mathscr  F(\alpha,u) = 0$ satisfying $\omega_{G}(\alpha_k) \neq 0$
for some $(\alpha_k,0) \in \mathcal U \cap \Lambda$, then either
\begin{enumerate}
\item[$(a)$] $\mathscr C \cap \partial \mathcal U \neq \emptyset$ (the branch reaches the boundary); or
\item[$(b)$] there exists a finite set $\mathscr C \cap \Lambda = \{(\alpha_{k_0},0),\ldots,(\alpha_{k_l},0)\}$ satisfying
$\sum_{i=0}^l \omega_{G}(\alpha_{k_i}) = 0$.
\end{enumerate}
\end{theorem}
To exclude alternative $(b)$ and thereby guarantee unboundedness, we note that the continuous dependence of $\sigma(\mathscr  A(\alpha))$ on $\alpha$ implies a telescoping of the sum of local bifurcation invariants. Specifically, by applying the computational formula for the degree at each regular endpoint (cf. \cite{Ghanem_annales}, or the proof of Proposition 4.5 in \cite{Ghanem2024}), one obtains:
\begin{align}\label{eq:telescope}
\sum_{k=0}^n \omega_{G}(\alpha_k) = G\text{-deg}(\mathscr  A(\alpha_0^-), B(\mathscr H)) - G\text{-deg}(\mathscr  A(\alpha_n^+), B(\mathscr H)).
\end{align}
The right-hand side of \eqref{eq:telescope} compares the spectral data strictly at the two extreme regular values $\alpha_0^-$ and $\alpha_n^+$, flanking the entire critical set $\Lambda$. By evaluating this difference using Proposition~\ref{prop:spectral_flow_evaluation}, we arrive at our main global result:
\begin{theorem}[Global Bifurcation for Ize Pairs]\label{thm:global_bif_ize}
Assume the critical set $\Lambda$ is finite as in \ref{bB}, and bounded within the interval $(\alpha_0^-, \alpha_n^+)$ where $\alpha_0^-$ and $\alpha_n^+$ are regular parameter values. If $(G, \mathbb V_{[\alpha_0^-,\alpha_n^+]})$ is an \textbf{Ize pair} with respect to a maximal orbit type $(H) \in \Phi_0(G; \mathbb V_{[\alpha_0^-,\alpha_n^+]})$, then there exists an \textbf{unbounded} branch $\mathscr C_H$ of solutions to \eqref{eq:bifurcation_problem} with symmetries at least $(H)$.
\end{theorem}
\begin{proof}
To prove the existence of an unbounded branch, we must exclude alternative $(b)$ in the Rabinowitz Alternative (Theorem~\ref{thm:Rab}) for at least one branch. Taking advantage of the telescoping property \eqref{eq:telescope}, the total sum of the local invariants across all critical points simplifies to the difference of the degrees at the endpoints. Setting $a = \alpha_0^-$ and $b = \alpha_n^+$, the global spectral flow representation $\mathbb V_{[a,b]}$ from Proposition~\ref{prop:spectral_flow_evaluation} is exactly $\mathbb V_{[\alpha_0^-,\alpha_n^+]}$. Since $(G, \mathbb V_{[\alpha_0^-,\alpha_n^+]})$ is an Ize pair for maximal $(H)$, Proposition~\ref{prop:spectral_flow_evaluation} guarantees that the coefficient of $(H)$ in the total sum $\sum_{k=0}^n \omega_{G}(\alpha_k)$ is non-zero. 

If every branch bifurcating from $\Lambda$ were bounded, alternative $(b)$ would apply to each of them. This would effectively partition the critical set $\Lambda$ into disjoint subsets over which the local invariants sum to zero, forcing the total sum over all of $\Lambda$ to be zero. Since we have established that the total sum is strictly non-zero, at least one branch $\mathscr C_H$ must violate alternative $(b)$ and therefore be unbounded. By Theorem~\ref{thm:abstract_local_bif}, this branch admits symmetries at least $(H)$.
\end{proof}
\begin{remark}\label{rem:infinite_lambda}\rm
The formulation of Theorem~\ref{thm:global_bif_ize} highlights the macroscopic nature of this approach. Notice that Proposition~\ref{prop:spectral_flow_evaluation} evaluates the total topological jump across the interval $[a,b]$ strictly through the boundary data $E^-(a)$ and $E^-(b)$. Consequently, the global bifurcation guarantee does not require a granular, point-by-point analysis of the individual critical kernels within the interval. As long as the interval boundaries are regular and the net spectral flow representation $\mathbb V_{[a,b]}$ constitutes an Ize pair, the topological change is strictly non-zero. This forces a branch to escape the local parameter window regardless of the internal complexity of the eigenvalue crossings within the critical set.
\end{remark}
\section{Equivariant Bifurcation of Periodic Solutions}
\label{sec:bifurcation_application}
The abstract framework of Sections \ref{sec:maximality}--\ref{sec:bifurcation_setup} reduces equivariant bifurcation guarantees to a parity check on the spectral flow representation. We now demonstrate that this reduction yields concrete, computable results by applying it to a class of problems where traditional Burnside ring methods are particularly unwieldy: symmetrically coupled oscillator networks.
\vs
Systems of symmetrically coupled oscillators arise throughout the physical sciences---from coupled Josephson junctions and laser arrays to molecular vibrations and neural networks---and serve as a
paradigmatic testing ground for equivariant topological methods. The key feature of such systems is that the coupling symmetry group combined with the time-translation and time-reversal symmetries inherent in periodic, autonomous second-order equations, induces a rich product group whose isotypic structure can be exploited by the degree-theoretic machinery of Section~\ref{sec:bifurcation_setup}.  We
now demonstrate this by applying the Ize pair framework to a concrete class of coupled oscillator equations.
\vs
Let $\Gamma$ be a finite group acting on $V := \br^N$ by permutation of coordinates:
\[
  \gamma(x_1, \ldots, x_N)
  := (x_{\gamma(1)}, \ldots, x_{\gamma(N)}).
\]
This action models a network of $N$ identical oscillators whose coupling topology is invariant under the symmetry group~$\Gamma$---for instance, $\Gamma = D_N$ for a ring of oscillators, or
$\Gamma = S_N$ for all-to-all coupling.  We study the parametrized system of $p$-periodic, second-order autonomous equations
\begin{equation}\label{eq:system_bif}
  \begin{cases}
    {\red -}\ddot x(t) = f(x(t)) + A(\alpha)\,x(t),
      \quad t \in \br,\; x(t) \in V, \\
    x(t) = x(t + p),\quad
    \dot x(t) = \dot x(t + p),
  \end{cases}
\end{equation}
where $\alpha \in \br$ is a bifurcation parameter,
$A\colon \br \to L^\Gamma(V)$ is an analytic family of $\Gamma$-equivariant symmetric matrices governing the linear coupling, and $f\colon V \to V$ is a continuous nonlinearity satisfying:
\begin{enumerate}[label=($A_\arabic*$)]
  \item\label{bA1} $f$ is $\Gamma$-equivariant:
    $f(\gamma x) = \gamma f(x)$ for all $\gamma \in \Gamma$,
    $x \in V$.
  \item\label{bA2} $f$ is odd: $f(-x) = -f(x)$ for all $x \in V$.
  \item\label{bA3} $f(x) = o(|x|)$ as $x \to 0$.
  \item\label{bA4} There exists $M > 0$ such that
    $f(x) \cdot x > 0$ for all $|x| \geq M$.
\end{enumerate}
Condition~\ref{bA1} ensures that the nonlinearity respects the network symmetry, while~\ref{bA2} introduces a $\bz_2$-antipodal symmetry reflecting the physical reversibility of the restoring force.
Together, \ref{bA1}--\ref{bA2} guarantee that
\eqref{eq:system_bif} admits the symmetries of the product group $G := O(2) \times \Gamma \times \bz_2$, where the $O(2)$-factor encodes the autonomous time-translation and time-reversal symmetries and the $\bz_2$-factor encodes the antipodal symmetry of the nonlinearity.  Conditions~\ref{bA3}--\ref{bA4} are analytic
hypotheses: the former ensures that the trivial equilibrium $x \equiv 0$ is a solution for all parameter values (enabling linearization),
and the latter provides a superlinear restoring force at large amplitudes that confines solutions to bounded sets---a compactness condition essential for the degree-theoretic framework.
\subsection{Functional Reformulation} \label{sec:functional_reformulation}
Let $\mathscr H := H^2(S^1; V)$ be the Sobolev space of $2\pi$-periodic, $V$-valued functions equipped with the standard inner product and the isometric $G$-action
\begin{align*}
(e^{i\theta}, \gamma, \pm 1)u(t) := \pm \gamma u(t + \theta), \quad 
(\kappa, \gamma, \pm 1)u(t) := \pm \gamma u(-t).
\end{align*}
Following the construction in \cite{Ghanem_annales} (cf. also \cite{Carlos2019}), the problem of finding $p$-periodic solutions to \eqref{eq:system_bif} is equivalent to solving operator equation
\begin{align}\label{eq:operator_equation}
\mathscr F(\alpha,u) = 0, \quad (\alpha,u) \in \br \times \mathscr H,    
\end{align}
where $\mathscr  F: \br \times \mathscr H \rightarrow \mathscr H$ is the family of compact perturbations of the identity
\begin{align} \label{def:operator_F}
\mathscr  F(\alpha,u) := u - \mathscr L^{-1}(\beta^2 N(j(u)) + \beta^2 A(\alpha) j(u) + j(u)),
\end{align}
defined in terms of the frequency $\beta:= p/(2\pi)$, the shifted Laplacian operator
\begin{align}\label{def:operator_L}
    \mathscr L: \mathscr H \rightarrow L^2(S^1;V), \quad \mathscr L u := {\red -}\ddot u + u,
\end{align}
the Nemytskii operator  $N:L^2(S^1;V) \rightarrow L^2(S^1;V)$, given by
$(Nu)(t) :=  f(u(t))$, and the compact Sobolev embedding $j:\mathscr H \hookrightarrow L^2(S^1;V)$.  The linearization of $\mathscr F$
at the origin is
\begin{align*}
\mathscr  A(\alpha): \mathscr H \rightarrow \mathscr H, \quad \mathscr  A(\alpha)u := u - \mathscr L^{-1}(\beta^2 A(\alpha) u + u).
\end{align*}
We denote the critical set of \eqref{eq:operator_equation} by
$\Lambda := \{(\alpha,0) \in \br \times \mathscr H :
\mathscr A(\alpha) \text{ is not an isomorphism}\}$ and, at every isolated critical point $(\alpha_0,0) \in \Lambda$ flanked by regular values $\alpha_0^- < \alpha_0 < \alpha_0^+$, define the local
bifurcation invariant by
\[
\omega_{G}(\alpha_0) := \gdeg(\mathscr A(\alpha_0^-),B(\mathscr H)) - \gdeg(\mathscr A(\alpha_0^+),B(\mathscr H)).
\]
\subsubsection{The $G$-Isotypic Decomposition of $\mathscr H$}\label{sec:G_isotypic_decomp}
Let $\operatorname{Irr}(\Gamma) = \{\mathcal V_j\}_{j=0}^r$ be a complete list of the irreducible $\Gamma$-representations and consider the $\Gamma \times \bz_2$-isotypic decomposition of the state space
\[
V \simeq \bigoplus_{j=0}^r V_j, \quad V_j \simeq \ell_j \mathcal V_j^-,
\]
where the superscript indicates the antipodal $\bz_2$-action on the irreducible $\Gamma$-representation $\mathcal V_j$ and $\ell_j$ is the $\Gamma$-isotypic multiplicity of the $j$-th $\Gamma$-isotypic component $V_j$. Since $A(\alpha) : V \to V$ is $\Gamma$-equivariant for all parameter values $\alpha \in \br$, it admits a block matrix decomposition of the form
\[
A(\alpha) = \bigoplus_{j=0}^r A_j(\alpha), \quad A_j(\alpha) := A|_{V_j}(\alpha) : V_j \to V_j.
\]
In particular, on each $\Gamma$-isotypic component $V_j$, the spectrum of the restriction $A_j(\alpha): V_j \to V_j$ consists of a finite number of real-valued, analytic (by Rellich's theorem) eigenvalues
\[
\sigma(A_j(\alpha)) = \{ \zeta_{j,1}(\alpha), \ldots, \zeta_{j,\ell_j}(\alpha) \}, \quad \zeta_{j,k} : \br \to \br.
\]
To ensure that the bifurcation parameter effectively modifies the spectrum of the linearization, we impose the following assumption
\begin{enumerate}[label=$(A_{\arabic*})$, start=0]
    \item\label{a0} The dependence of the coupling matrix $A(\alpha)$ on the bifurcation parameter $\alpha \in \br$ is such that the eigenvalues $\zeta_{j,k}(\alpha)$ are non-constant functions.
\end{enumerate}
We denote by $\mathcal W_0 \simeq \br$ the trivial irreducible $O(2)$-representation and for each $m \in \bn$ by $\mathcal W_m \simeq \bc$ the irreducible $O(2)$-representation equipped with the $m$-folded $O(2)$-action
\[
\theta w := e^{im\theta} \cdot w, \quad \kappa w := \overline{w}, \quad w \in \mathcal W_m,
\]
where `$\cdot$' is the standard complex multiplication and `$\overline{\phantom{w}}$' is complex conjugation. With this notation, we can identify each irreducible $G$-representation $\mathcal V_{m,j} := \mathcal W_m \otimes \mathcal V_{j}^-$ with an index pair $(m,j) \in \bn \cup \{0\} \times \{0,\ldots,r\}$. Since every $u \in \mathscr H$ has a Fourier decomposition of the form
\[
u(t) = \sum_{m=0}^\infty \cos(mt)a_m + \sin(mt)b_m, \quad a_m,b_m \in V,
\]
the Laplacian operator \eqref{def:operator_L} admits the spectrum $\sigma(\mathscr L) = \{m^2+1 : m \in \bn \cup \{0\} \}$ with the associated eigenspaces $\mathscr H_{m} := \{\cos(mt)a + \sin(mt)b : a,b \in V \}$. Clearly, one has $\mathscr H_{m} \simeq \mathcal W_m \otimes V$ for all $m \in \bn \cup \{0\}$. Notice also that $\mathscr H_m \subset \mathscr H$ admits the $\Gamma$-isotypic decomposition
\[
\mathscr H_m = \bigoplus_{j=0}^r \mathscr H_{m,j}, \quad \mathscr H_{m,j}:= \{\cos(mt)a + \sin(mt)b : a,b \in V_j \},
\]
with $\mathscr H_{m,j} \simeq \ell_j \mathcal V_{m,j}$, for each $(m,j) \in \bn \cup \{0\} \times \{0,\ldots,r\}$. Therefore, the $G$-isotypic decomposition of $\mathscr H$ is given by
\begin{align} \label{eq:G_isotypic_decomp}
   \mathscr H = \overline{\bigoplus_{j=0}^r \bigoplus\limits_{m=0}^\infty \mathscr H_{m,j}},
\end{align}
where the closure is taken in $\mathscr H$. As a $G$-equivariant linear operator, $\mathscr A(\alpha):\mathscr H \rightarrow \mathscr H$ respects the decomposition \eqref{eq:G_isotypic_decomp} in the sense that
$\mathscr A(\alpha)(\mathscr H_{m,j}) \subset \mathscr H_{m,j}$ for all $(m,j) \in \bn \cup \{0\} \times \{0,\ldots,r\}$. Adopting the notation $\mathscr A_{m,j}(\alpha) := \mathscr A(\alpha)|_{\mathscr H_{m,j}}: \mathscr H_{m,j} \rightarrow \mathscr H_{m,j}$, it follows that the spectrum of $\mathscr A(\alpha)$ is given by
\begin{align*} 
 \sigma (\mathscr A(\alpha))= \bigcup\limits_{m=0}^{\infty} \bigcup_{j=0}^r \sigma(\mathscr A_{m,j}(\alpha)), \quad \sigma(\mathscr A_{m,j}(\alpha)) = \left\{ \frac{m^2 - \beta^2 \zeta_{j,k}(\alpha)}{1+m^2} : \zeta_{j,k}(\alpha) \in \sigma(A_j)\right\}.
\end{align*}
We identify the eigenvalues of $\mathscr A(\alpha)$ with the notation
\[
\mu_{m,j,k}(\alpha) := \frac{m^2 - \beta^2 \zeta_{j,k}(\alpha)}{1+m^2},
\]
noticing that $(\alpha_0,0) \in \Lambda$ if and only if there exists an index triple $(m,j,k) \in \bn \cup \{0\} \times \{0,\ldots,r\} \times \{1,\ldots,\ell_j\}$ satisfying
\[
\alpha_0 =\zeta_{j,k}^{-1}(m^2/\beta^2).
\]
\begin{lemma}
Under assumption \ref{a0}, the critical set $\Lambda$ is discrete.
\end{lemma}
\begin{proof}
Since $m^2/\beta^2 \to \infty$ as $m \to \infty$, any bounded interval $I \subset \br$ contains critical points from only finitely many Fourier modes.  For a fixed mode~$m$, the function
$\psi_{m,j,k}(\alpha) := \zeta_{j,k}(\alpha) - m^2/\beta^2$ is analytic and, by~\ref{a0}, not identically zero. Its zero set is therefore discrete.
\end{proof}
\subsection{Local and Global Bifurcation of Periodic Solutions}\label{sec:computation_lbi}

To apply the abstract results of
Section~\ref{sec:bifurcation_setup}, we translate the Ize pair hypothesis into a computable spectral condition. For any orbit type $(H) \in \Phi_0(G; \mathscr H)$, define the set of \emph{geometrically compliant} indices
\[
\mathcal I_H := \{(m,j) : \dim \mathcal V_{m,j}^H \equiv 1 \pmod 2\}.
\]
Given a parameter interval $[a,b]$ with regular end points, the spectral flow representation $\mathbb V_{[a,b]} \simeq \bigoplus_{j=0}^r \bigoplus_{m=0}^\infty \rho_{m,j} \mathcal V_{m,j}$ constitutes an Ize pair with respect to a maximal orbit type $(H) \in \Phi_0(G; \mathbb V_{[a,b]})$ if and only if  the total crossing parity across compliant $G$-isotypic indices is odd:
\begin{equation}\label{eq:ize_parity_condition}
\sum_{(m,j) \in \mathcal I_H} \rho_{m,j} \equiv 1 \pmod 2,
\end{equation}
where $\rho_{m,j} \in \{0,1\}$ is the net parity of negative eigenvalues crossing zero in the $G$-isotypic component $\mathscr H_{m,j}$ across the interval $[a,b]$. Notice that, since the critical set $\Lambda$ is discrete, this sum always admits a finite number of non-zero terms. We leverage this characterization to formulate the following bifurcation guarantees.
\begin{theorem}[Local Bifurcation of Periodic Solutions]\label{thm:local_bif_main}
Let $(\alpha_0,0) \in \Lambda$ be any critical point of the periodic system \eqref{eq:operator_equation} with the deleted regular neighborhood $[\alpha_0^-, \alpha_0^+] \setminus \{\alpha_0\}$. If there exists a maximal orbit type $(H) \in \Phi_0(G; \mathbb V_{[\alpha_0^-,\alpha_0^+]})$ satisfying the parity condition \eqref{eq:ize_parity_condition} across the interval $[\alpha_0^-, \alpha_0^+]$, then there exists a branch $\mathscr C_H$ of non-trivial $p$-periodic solutions to \eqref{eq:system_bif} bifurcating from $(\alpha_0,0)$ with symmetries at least $(H)$.
\end{theorem}
\begin{proof}
The result follows as a direct application of Theorem~\ref{thm:local_bif_ize}. Since the completely continuous vector field $\mathscr F$ defined in \eqref{def:operator_F} satisfies the abstract hypotheses outlined in Section~\ref{sec:bifurcation_setup}, the local topological jump across the critical point is completely captured by the spectral flow representation. The parity condition \eqref{eq:ize_parity_condition} is equivalent to the Ize pair hypothesis via Theorem~\ref{thm:parity_guarantee}, guaranteeing that $\operatorname{coeff}^H(\omega_G(\alpha_0)) \neq 0$. Krasnosel'skii's Theorem (Theorem~\ref{thm:abstract_local_bif}) then ensures the emergence of the $H$-symmetric branch of non-trivial periodic solutions.
\end{proof}

\subsubsection{Fixed-Point Reduction}\label{sec:fixed_point_reduction}
Theorem~\ref{thm:local_bif_main} guarantees a bifurcating branch of non-trivial periodic solutions, but a priori this branch could include constant (stationary) equilibria---functions $u(t) \equiv x_0 \neq 0$ that satisfy the periodicity condition trivially. To obtain genuinely oscillatory solutions, we restrict the problem to a subspace that excludes constants by design.
\vs
Consider the subgroup $\bm K := \{(-1,-1),(1,1)\} \leq O(2) \times \bz_2 \leq G$. A function $u \in \mathscr H^{\bm K}$ satisfies the half-period anti-symmetry $u(t+\pi) = -u(t)$, which precludes constant solutions. The $\bm K$-fixed-point space admits the isotypic decomposition
\begin{align*}
\mathscr H^{\bm K} = \overline{\bigoplus_{j=0}^r\bigoplus\limits_{m \in 2\bn-1} \mathscr H_{m,j}},
\end{align*}
consisting only of the odd Fourier mode components. This space is an isometric Hilbert representation of the group $\bm G := N(\bm K)/\bm K \simeq O(2) \times \Gamma $. We denote the restricted bifurcation map and linearization by
\[
\mathscr  F^{\bm K} := \mathscr  F|_{\br \times \mathscr H^{\bm K}}: \br \times \mathscr H^{\bm K} \to \mathscr H^{\bm K}, \quad \mathscr  A^{\bm K}(\alpha) := \mathscr  A(\alpha)|_{\mathscr H^{\bm K}}:\mathscr H^{\bm K} \to \mathscr H^{\bm K},
\]
with the $\bm K$-fixed critical set $\Lambda^{\bm K} := \{(\alpha,0) \in M : \mathscr  A^{\bm K}(\alpha) \text{ is singular}\}$. The local bifurcation invariant in this restricted setting is
\begin{align*}
\omega_{\bm G}(\alpha_0) := \bm G\text{-deg}(\mathscr  A^{\bm K}(\alpha_0^-), B(\mathscr H^{\bm K})) - \bm G\text{-deg}(\mathscr  A^{\bm K}(\alpha_0^+), B(\mathscr H^{\bm K})),
\end{align*}
and Theorem~\ref{thm:abstract_local_bif} applies with $G$ replaced by $\bm G$ and $\mathscr H$ by $\mathscr H^{\bm K}$, with the advantage that any branch of nontrivial solutions to $\mathscr  F^{\bm K}(\alpha,u) = 0$ consists exclusively of non-stationary solutions.

\begin{theorem}[Global Bifurcation of Periodic Solutions]\label{thm:global_bif_main}
Let $[a, b]$ be a parameter interval with regular endpoints containing at least one critical parameter value $a < \alpha_0 < b$ satisfying $\omega_{\bm G}(\alpha_0) \neq 0$. If there exists a maximal orbit type $(H) \in \Phi_0(\bm G; \mathbb V_{[a,b]})$ satisfying the parity condition \eqref{eq:ize_parity_condition} across the global interval $[a,b]$, then there exists a branch $\mathscr C_H$ of strictly non-stationary, $p$-periodic solutions with symmetries at least $(H)$ that bifurcates from the trivial branch inside $(a,b)$ and satisfies $\mathscr{C}_H \not\subset (a,b) \times B_R(\mathscr{H}^{\bm{K}})$ for any $R > 0$.
\end{theorem}
\begin{proof}
Since the critical set is discrete, its intersection with the bounded interval $[a,b]$ is strictly finite. Restricting the operator to the $\bm K$-fixed point space $\mathscr H^{\bm K}$, we apply Proposition~\ref{prop:spectral_flow_evaluation} relative to the restricted group $\bm G$. The parity condition \eqref{eq:ize_parity_condition} ensures that the global spectral flow representation constitutes an Ize pair, which guarantees that the $H$-coefficient in the total sum of local bifurcation invariants $\sum_{\alpha_k \in (a,b)} \omega_{\bm G}(\alpha_k)$ is strictly non-zero. 

For any arbitrary radius $R > 0$, consider the open bounded invariant cylinder $\mathcal U := (a,b) \times B_R(\mathscr H^{\bm K})$. By the Rabinowitz Alternative (Theorem~\ref{thm:Rab}), the non-zero topological degree sum strictly rules out the possibility of the branch $\mathscr C_H$ forming a closed loop within $\mathcal U$. Consequently, $\mathscr C_H$ must intersect the boundary $\partial \mathcal U$. This implies that the branch either grows unbounded in amplitude (intersecting $[a,b] \times \partial B_R(\mathscr H^{\bm K})$ for arbitrarily large $R$) or it crosses the parameter bounds (intersecting $\{a,b\} \times \mathscr H^{\bm K}$). 

Finally, since every element $u \in \mathscr H^{\bm K}$ satisfies the half-period anti-symmetry $u(t + p/2) = -u(t)$, the only constant solution satisfying this property is the trivial solution $u \equiv 0$. The branch is thereby protected from terminating at any non-trivial stationary equilibrium, ensuring the solutions remain strictly non-stationary globally.
\end{proof}

\subsection{Methodological Comparison: The $\Gamma = D_8$ Symmetric Network}
\label{sec:methodological_comparison}
To illustrate the computational advantage of the Ize pair framework in a general bifurcation setting, we contrast it directly with the standard basic-degree pipeline utilized in recent literature (cf.\ \cite{Ghanem_annales}). Consider a general parameter crossing in a $D_8$-symmetric network of coupled oscillators, where the full symmetry group is $G = O(2) \times D_8 \times \bz_2$. 
\vs
To specify the relevant symmetries, recall that the irreducible $G$-representations take the form $\mathcal{V}_{m,j} := \mathcal{W}_m \otimes \mathcal{V}_j^-$, where $m \in \bn \cup \{0\}$ denotes the temporal Fourier mode and $j \in \{0, 1, 2, 3, 4\}$ indexes the irreducible spatial representations of $D_8$. Specifically, $\mathcal{V}_0 \simeq \br$ is the trivial (fully synchronous) representation, $\mathcal{V}_1, \mathcal{V}_2, \mathcal{V}_3 \simeq \bc$ are the standard two-dimensional geometric representations, and $\mathcal{V}_4 \simeq \br$ is the one-dimensional alternating representation.
\vs
Suppose the bifurcation parameter $\alpha$ crosses an isolated critical value $\alpha_0$, resulting in the following shift in the negative eigenspaces at the flanking regular parameter values $\alpha_0^- < \alpha_0 < \alpha_0^+$:
\[
E(\alpha_0^-) \simeq \mathcal{V}_{1,0} \oplus \mathcal{V}_{1,1}, \quad \text{ and } \quad E(\alpha_0^+) \simeq \mathcal{V}_{1,0} \oplus \mathcal{V}_{1,2}.
\]
In this scenario, the irreducible $G$-representation $\mathcal{V}_{1,0}$ acts as a ``spectator" mode, while $\mathcal{V}_{1,1}$ and $\mathcal{V}_{1,2}$ cross the origin in opposite directions. We wish to guarantee the emergence of a non-stationary branch with the maximal orbit type $(H) \in \Phi_0(G)$. Using the amalgamated notation for product subgroups (see Appendix \ref{sec:product_geometry}), we examine the maximal subgroup:
\[
  H := D_{4} \prescript{\bz_1}{}\times_{D_4}^{\bz_4^d} D_8^p,
\]
where if $D_8$ is generated by the rotation 
$\gamma := e^{\frac{2 \pi i}{8}}$ and the relection $\kappa$, we have adopted the following ancillary shorthand $\bz_4^d := \{ (1,1), (\gamma^2,-1), (\gamma^4,1), (\gamma^6,-1) \} \leq D_8 \times \bz_2$ and $D_8^p := D_8 \times \bz_2$.
\vs
\noi \textbf{The Basic-Degree Pipeline:} 
In the standard equivariant degree framework, the topological change is captured by evaluating the local bifurcation invariant as the difference of two degrees in the Burnside ring $A(G)$. For this eigenvalue crossing, the invariant is the difference of the respective basic degree products:
\[
  \omega_G(\alpha_0) = \left( \deg_{\mathcal{V}_{1,0}} \cdot \deg_{\mathcal{V}_{1,1}} \right) - \left( \deg_{\mathcal{V}_{1,0}} \cdot \deg_{\mathcal{V}_{1,2}} \right) \;\in A(G).
\]
Determining whether $(H)$ bifurcates requires extracting the integer coefficient $\operatorname{coeff}^H(\omega_G(\alpha_0))$. Even though $\mathcal{V}_{1,0}$ does not cross the origin, its basic degree must be carried through the multiplication on both sides before the subtraction can be resolved. Since Burnside ring multiplication of involutive elements requires evaluating recursive sums across a dense isotropy lattice, this task is highly non-trivial. A GAP routine (see \cite{GAP} for information on the EquiDeg GAP package) to verify this coefficient requires instantiating the group and resolving the algebra computationally:
\begin{lstlisting}[language=GAP, caption={GAP routine to compute the local bifurcation invariant $\omega_G(\alpha_0)$ in the Burnside ring $A(O(2) \times D_8 \times \mathbb{Z}_2)$.}, label={lst:gap_basic_degree}]
LoadPackage("EquiDeg");
# Instantiate G = O(2) x D_8 x Z_2
o2:=OrthogonalGroupOverReal(2);
d8:=pDihedralGroup(8);
z2:=pCyclicGroup(2);
# generate D8xZ2
g1:=DirectProduct(d8,z2);
# set names for subgroup conjugacy classes in D8xZ2
SetCCSsAbbrv(g1,["Z1","Z2","D1t","D1z","D1","Z1m","Z1p","D1zt","D1pt","D1p","D2","Z4","D2t","D2zt","D2d","Z4d","D2dt","D2z","Z2p","Z4p","D4dt","D2p","D4","D2pt","D4z","D4zt","Z8","Z8d","D4t","D4d","D4p","Z8p","D8","D4pt","D8d","D8z","D8dt","D8p"]);
# generate O(2)xD8xZ2
G:=DirectProduct(o2,g1);
ccs:=ConjugacyClassesSubgroups(G);
irrs := Irr(G);

# Compute the basic degrees for the relevant isotypic components
deg0 := BasicDegree(irrs[1,1]);  # V_{1,0} (Spectator mode)
deg1 := BasicDegree(irrs[1,9]);  # V_{1,1}
deg2 := BasicDegree(irrs[1,11]); # V_{1,2}

# Compute the local bifurcation invariant in the Burnside ring
omega := (deg0 * deg1) - (deg0 * deg2);
\end{lstlisting}
Running Listing \ref{lst:gap_basic_degree} confirms that $\operatorname{coeff}^H(\omega_G(\alpha_0)) = -2 \neq 0$, successfully guaranteeing bifurcation via the standard basic degree pipeline.

\vs
\noi \textbf{The Spectral Flow Pipeline:}
Under the framework developed in this paper, we bypass the Burnside ring multiplication entirely by tracking the net topological jump. We begin by constructing the mod-2 spectral flow representation (Definition~\ref{def:spectral_flow_representation}):
\[
  \mathbb{V}_{[\alpha_0^-,\alpha_0^+]} = \bigoplus_{i \in \mathcal{I}} \rho_i \mathcal{V}_{1,i}.
\]
Since the spectator mode $\mathcal{V}_{1,0}$ remains negative across the interval, its crossing parity is $\rho_0 \equiv 1 - 1 \equiv 0 \pmod{2}$. The changing modes $\mathcal{V}_{1,1}$ and $\mathcal{V}_{1,2}$ admit parities of $1$, reducing the spectral flow representation to:
\[
\mathbb{V}_{[\alpha_0^-,\alpha_0^+]} = \mathcal{V}_{1,1} \oplus \mathcal{V}_{1,2}.
\]
To guarantee bifurcation for $(H)$, Theorem~\ref{thm:local_bif_main} requires only that $(G, \mathbb{V}_{[\alpha_0^-,\alpha_0^+]})$ constitutes an Ize pair. By the Parity Criterion (Theorem~\ref{thm:parity_guarantee}), this reduces strictly to checking the parity of the difference in fixed-point dimensions on this spectral flow representation. Since the representation contains only odd temporal modes, the trivial representation $\mathcal V_{0,0} \simeq \br$ is absent, making $\dim \mathbb V_{[\alpha_0^-,\alpha_0^+]}^G = 0$. Thus, we simply compute the parity of the $H$-fixed-point space:
\[
  \dim \mathbb{V}_{[\alpha_0^-,\alpha_0^+]}^H - \dim \mathbb{V}_{[\alpha_0^-,\alpha_0^+]}^G = \dim \mathcal{V}_{1,1}^H + \dim \mathcal{V}_{1,2}^H - 0.
\]
From the fixed-point data geometrically verified for $D_8$ (cf.\ \cite{Ghanem_annales}), we know that $(H)$ is maximal in $\mathcal{V}_{1,1}$ with $\dim \mathcal{V}_{1,1}^H = 1$. Conversely, $(H)$ is absent from the isotropy lattice of $\mathcal{V}_{1,2}$, meaning its fixed-point space on that component is strictly trivial ($\dim \mathcal{V}_{1,2}^H = 0$). Therefore, the total fixed-point dimension is $1 + 0 = 1 \equiv 1 \pmod{2}$. The Parity Criterion is met, and bifurcation is guaranteed. The spectator mode $\mathcal{V}_{1,0}$ is algebraically eliminated before the geometric evaluation begins, and the entire algebraic obstruction is resolved by a single dimension count. Basic degree computations, recurrence formulas, and Burnside ring arithmetic are rendered entirely unnecessary.

\subsection{Constructing the Compliant Index Set $\mathcal{I}_H$} \label{sec:compliant_set}

With the structure of amalgamated fixed-point spaces established, we return to the specific functional setting of our periodic bifurcation problem. Recall from Section~\ref{sec:computation_lbi} that the detection of global equivariant bifurcation at a maximal orbit type $(H)$ hinges on identifying the \textit{compliant index set}:
\[
\mathcal I_H := \{(m,j) : \dim \mathcal V_{m,j}^H \equiv 1 \pmod 2\}.
\]
Recall that every isotropy subgroup $H \leq O(2) \times (\Gamma \times \bz_2)$ admits an amalgamated decomposition (see Appendix \ref{sec:product_geometry}) of the form
\[
H = K_O \prescript{\varphi_O}{}\times_{L}^{\varphi_\Gamma} K_\Gamma,
\]
with projection kernels $Z_O := \ker \varphi_O \leq K_O$ and $Z_\Gamma := \ker \varphi_\Gamma \leq K_\Gamma$. Applying Theorem~\ref{thm:amalgamated_fp} to our state space, the $H$-fixed-point space of a specific spatio-temporal mode decomposes as:
\begin{equation*} 
\mathcal V_{m,j}^H = (\mathcal W_m \otimes \mathcal V_{j}^-)^H \simeq \left( \mathcal W_m^{Z_O} \otimes (\mathcal V_{j}^-)^{Z_\Gamma} \right)^L.
\end{equation*}
To systematically construct $\mathcal{I}_H$ without relying on computational group theory software, one evaluates the geometric compatibility between the spatial symmetries of the network and the temporal frequencies of the dynamical system. By Corollary~\ref{cor:dim_parity}, this reduces to analyzing the $L$-isotypic multiplicities of the kernel-fixed-point spaces. 
\vs
Setting $\operatorname{Irr}(L) = \{\mathcal{L}_i\}_{i \in \mathcal{I}_L}$, we decompose the temporal and spatial kernel-fixed spaces into their $L$-isotypic components:
\[
\mathcal W_m^{Z_O} \simeq \bigoplus_{i \in \mathcal{I}_L} a_{m,i} \mathcal{L}_i \quad \text{and} \quad (\mathcal V_{j}^-)^{Z_\Gamma} \simeq \bigoplus_{i \in \mathcal{I}_L} b_{j,i} \mathcal{L}_i,
\]
where the multiplicities $a_{m,i}$ and $b_{j,i}$ are non-negative integers. It follows directly from Corollary~\ref{cor:dim_parity} that the dimension of the $H$-fixed-point space is the scalar product of these multiplicity vectors, weighted by the endomorphism dimensions $d_i := \dim \operatorname{End}_L(\mathcal{L}_i)$:
\begin{equation} \label{eq:compliant_set_condition}
\dim \mathcal V_{m,j}^H = \sum_{i \in \mathcal{I}_L} a_{m,i} b_{j,i} d_i.
\end{equation}
Consequently, the compliant index set $\mathcal{I}_H$ consists of those Fourier modes $m \in \bn \cup \{0\}$ and spatial isotypic indices $j \in \{0, 1, \dots, r\}$ that satisfy the resonance condition:
\begin{equation*}
\mathcal I_H = \left\{ (m,j) : \sum_{i \in \mathcal{I}_L} a_{m,i} b_{j,i} d_i \equiv 1 \pmod 2 \right\}.
\end{equation*}

\begin{example}[Explicit Construction of $\mathcal{I}_H$] \label{ex:explicit_compliant_set}
To demonstrate the systematic construction of the compliant set using the theory developed in Section~\ref{sec:compliant_set}, let us explicitly evaluate $\mathcal I_H$ for the maximal orbit type:
\[
  H = D_{4} \prescript{\bz_1}{}\times_{D_4}^{\bz_4^d} D_8^p \;\leq\; O(2) \times (D_8 \times \bz_2).
\]
Here, the projection subgroups are $K_O = D_4$ and $K_\Gamma = D_8 \times \bz_2$ (denoted in our shorthand as $D_8^p$). The respective kernels are $Z_O = \bz_1 = \{e\}$ and $Z_\Gamma = \bz_4^d = \langle (\gamma^2, -1) \rangle$, yielding the amalgamation quotient $L \simeq D_4$.

\vs
\noi \textbf{Step 1: The Spatial Kernel-Fixed Space.}
We first determine which spatial modes $j \in \{0, 1, 2, 3, 4\}$ survive the action of $Z_\Gamma$. For a vector $v \in \mathcal V_j^-$ to be fixed by the generator $z = (\gamma^2, -1)$, we require:
\[
  (\gamma^2, -1) \cdot v = - \rho_j(\gamma^2) v = v \implies \rho_j(\gamma^2) = -\id.
\]
Since $\gamma$ generates a spatial rotation by $2\pi/8$, the operator $\rho_j(\gamma^2)$ represents a rotation by $j (4\pi/8) = j \pi / 2$. For this to coincide with $-\id$ (a rotation by $\pi$), we must have $j \pi / 2 \equiv \pi \pmod{2\pi}$, implying $j \equiv 2 \pmod 4$. Thus, within our index range, $j=2$ is the \emph{only} spatial mode that admits a non-trivial fixed-point space. 

For $j=2$, $Z_\Gamma$ acts trivially, leaving $(\mathcal V_2^-)^{Z_\Gamma} = \mathcal V_2^-$. As a representation of the quotient $L \simeq D_4$, this forms the standard 2-dimensional geometric representation, which we denote $\mathcal{L}_{*}$. Therefore, the spatial multiplicity $b_{j,*} = 1$ for $j=2$, and $0$ otherwise.

\vs
\noi \textbf{Step 2: The Temporal Kernel-Fixed Space.}
Since $Z_O = \{e\}$, the temporal space $\mathcal W_m$ is entirely fixed by the temporal kernel, leaving $\mathcal W_m^{Z_O} = \mathcal W_m$. We must determine its $L$-isotypic decomposition under the action of $K_O = D_4$. The rotation generator of $D_4$ acts on time by shifting it by $T/4$, corresponding to a phase shift of $m \pi / 2$ in the $m$-th Fourier mode. 

To resonate with the spatial representation $\mathcal{L}_{*}$, this temporal phase shift must act as a primitive rotation by $\pm \pi/2$. This requires $m \pi / 2 \equiv \pm \pi/2 \pmod{2\pi}$, meaning $m$ must be odd. Thus, for any odd $m$, $\mathcal W_m \simeq \mathcal{L}_{*}$, yielding a temporal multiplicity $a_{m,*} = 1$. (For even $m$, the representation splits into 1-dimensional representations, yielding $a_{m,*} = 0$).
\vs
\noi \textbf{Step 3: Evaluating the Dimension.}
Applying the dimension formula (Equation~\ref{eq:compliant_set_condition}) to the non-zero multiplicities, we compute the dimension of the $H$-fixed-point space for any odd $m$ and $j=2$. Since $\mathcal{L}_{*}$ is absolutely irreducible over $\mathbb{R}$, Schur's Lemma guarantees that its endomorphism ring is strictly isomorphic to $\mathbb{R}$, yielding $\dim \operatorname{End}_L(\mathcal{L}_{*}) = 1$, such that
\[
\dim \mathcal V_{m,j}^H = a_{m,*} b_{j,*} \dim \operatorname{End}_L(\mathcal{L}_{*}) = (1)(1)(1) = 1.
\]
Since no other spatial modes survive the kernel projection, the compliant index set is explicitly defined as
\[
\mathcal I_H = \big\{ (m, 2) : m \in 2\bn - 1 \big\}.
\]
This algebraic reduction isolates the exact spatio-temporal harmonics where this specific maximal symmetry breaking can occur, requiring no computational search.
\end{example}
\vs
This concludes the main development. Appendix~\ref{sec:product_geometry} details the fixed-point geometry of product groups and establishes the dimensional parity formulas utilized in our explicit constructions. Appendix~\ref{sec:appendix} collects the axiomatic framework of the $G$-equivariant Leray--Schauder degree and proves the nontriviality criterion (Lemma~\ref{lemm:nontriviality_conditions}) invoked in the Introduction. Finally, Appendix~\ref{sec:appendix_esf} formalizes the mod-2 equivariant spectral flow within the representation ring $RO_2(G)$.
\appendix
\section{Fixed-Point Geometry in Product Groups} \label{sec:product_geometry}
To facilitate the identification of Ize pairs in the context of product groups $G := G_1 \times G_2$, we require a precise description of the subgroup lattice $\Phi_0(G)$ and the structure of fixed-point spaces for tensor product representations. We utilize the convention of \textit{amalgamated notation}, a consequence of Goursat's Lemma, which identifies subgroups of a product group via their projections.
\subsection{Amalgamated Subgroups}\label{sec:amalgamated_subgroups} Let $H$ be a closed subgroup of the product group $G := G_1 \times G_2$. The projections of $H$ onto the factors $G_1$ and $G_2$ define two subgroups $K_1 := \pi_1(H) \leq G_1$ and $K_2 := \pi_2(H) \leq G_2$. Furthermore, Goursat's Lemma implies that $H$ determines a pair of normal subgroups $Z_1 \unlhd K_1$ and $Z_2 \unlhd K_2$ such that the quotient groups are isomorphic, i.e., $K_1/Z_1 \simeq K_2/Z_2 \simeq L$. Consequently, $H$ is uniquely identified by the quintuple $(K_1,\varphi_1,K_2,\varphi_2,L)$ as follows:
\[
H := K_1^{\varphi_1} \times_L^{\varphi_2} K_2 = \{ (g_1, g_2) \in K_1 \times K_2 : \varphi_1(g_1) = \varphi_2(g_2) \},
\]
where $\varphi_1: K_1 \to L$ and $\varphi_2: K_2 \to L$ are epimorphisms with kernels $Z_1$ and $Z_2$, respectively. We refer to $L$ specifically as the \textit{amalgamation quotient} of $H$. Often, we are only interested in the conjugacy class $(H)$ of a subgroup $H \leq G$, which can be identified by the quadruple
$(K_1,Z_1,K_2,Z_2)$ as follows
\[
(H) := (K_1^{Z_1} \times^{Z_2} K_2).
\]

\subsection{Fixed-Point Spaces of Tensor Representations}
Let $\mathcal{U}$ be an orthogonal $G_1$-representation and $\mathcal{W}$ an orthogonal $G_2$-representation. The tensor product $\mathcal{V} := \mathcal{U} \otimes \mathcal{W}$ naturally becomes an orthogonal $G$-representation under the action $(g_1, g_2) \cdot (u \otimes w) := (g_1 u) \otimes (g_2 w)$. 
Given an amalgamated subgroup $H = K_1^{\varphi_1} \times_L^{\varphi_2} K_2$ (as defined in Section~\ref{sec:amalgamated_subgroups}), we would like to characterize its fixed-point space $\mathcal{V}^H$. To do so, we first establish that the amalgamation quotient $L$ inherits a natural action on the subspaces fixed by the kernels $Z_1 = \ker \varphi_1$ and $Z_2 = \ker \varphi_2$.
\begin{lemma}[Induced Quotient Action] \label{lem:induced_action}
Let $\mathcal{V}$ be a $G$-representation and $K \leq G$ a subgroup admitting an epimorphism $\varphi: K \to L$ with a kernel $Z \unlhd K$. The $Z$-fixed-point subspace $\mathcal{V}^{Z}$ is invariant under $K$, and the action of $K$ induces a  well defined $L$-action on $\mathcal{V}^{Z}$, denoted by $\wt{\varphi}: L \to GL(\mathcal{V}^{Z})$.
\end{lemma}
\begin{proof} First, we verify invariance. Take $u \in \mathcal{V}^{Z}$ and $k \in K$. For any kernel element $z \in Z$, the normality of $Z$ in $K$ implies $k^{-1} z k \in Z$. Thus, for any $(z,k,u) \in Z \times K \times \mathcal V^{Z}$, one has
\[
z \cdot (k \cdot u) = k \cdot (k^{-1} z k \cdot u)  = k \cdot u.
\]
Since $k \cdot u$ remains fixed by $Z$, the subspace $\mathcal{V}^{Z}$ is $K$-invariant. Next, we define the induced action. For any $l \in L$, choose a lift $k \in K$ such that $\varphi(k) = l$, and define
\[
\tilde{\varphi}(l)u := k \cdot u.
\]
To verify this is well-defined, suppose $k'$ is another lift satisfying $\varphi(k') = l$. Then $\varphi(k^{-1}k') = e_L$, which implies $k^{-1}k' \in Z$, or $k' = kz$ for some $z \in Z$. Since $u \in \mathcal{V}^{Z}$, we obtain
\[
k' \cdot u = (kz) \cdot u = k \cdot (z \cdot u) = k \cdot u.
\]
The action is therefore independent of the choice of lift.
\end{proof}
\vs
With this mechanism in place, we can state the main characterization result.
\begin{theorem}[Amalgamated Fixed-Point Formula] \label{thm:amalgamated_fp}Let $H = K_1^{\varphi_1} \times_L^{\varphi_2} K_2$ be a subgroup of $G_1 \times G_2$. The $H$-fixed-point space of the tensor representation $\mathcal{V} = \mathcal{U} \otimes \mathcal{W}$ is isomorphic to the subspace of $L$-invariants within the tensor product of the kernel-fixed subspaces:\begin{equation} \label{eq:fp_iso}(\mathcal{U} \otimes \mathcal{W})^H \simeq \left( \mathcal{U}^{Z_1} \otimes \mathcal{W}^{Z_2} \right)^L,\end{equation}where $L$ acts diagonally via the induced representations $\tilde{\varphi}_1$ and $\tilde{\varphi}_2$ defined in Lemma \ref{lem:induced_action}.\end{theorem}\begin{proof}The condition for a vector $v \in \mathcal{U} \otimes \mathcal{W}$ to be fixed by $H$ is that $(g_1, g_2) \cdot v = v$ for all $(g_1, g_2) \in H$. First, observe that $H$ contains the subgroups $Z_1 \times \{e\}$ and $\{e\} \times Z_2$. In particular, for any $z \in Z_1$, we have $\varphi_1(z) = e_L = \varphi_2(e)$, implying $(z, e) \in H$. Thus, any $H$-fixed vector must necessarily be fixed by $Z_1 \times Z_2$. The fixed-point space of a direct product acting on a tensor product decomposes as the tensor product of fixed-point spaces, i.e.
\[
(\mathcal{U} \otimes \mathcal{W})^{Z_1 \times Z_2} = \mathcal{U}^{Z_1} \otimes \mathcal{W}^{Z_2}.
\]
Now, consider the surjective homomorphism $\Psi: H \to L$ defined by $\Psi(g_1, g_2) = \varphi_1(g_1)$. This is well-defined because $\varphi_1(g_1) = \varphi_2(g_2)$ for all 
$(g_1, g_2) \in H$, so we can take the first projection without any loss of generality. 
For any pure tensor $v = u \otimes w \in \mathcal{U}^{Z_1} \otimes \mathcal{W}^{Z_2}$, the action of an element $h \in H$ is identical to the diagonal action of its image $l = \Psi(h) \in L$. Specifically, using the induced representations $\tilde{\varphi}_1$ and $\tilde{\varphi}_2$ from Lemma~\ref{lem:induced_action}, one has
\[
h \cdot (u \otimes w) = (g_1 \cdot u) \otimes (g_2 \cdot w) = (\tilde{\varphi}_1(l) \cdot u) \otimes (\tilde{\varphi}_2(l) \cdot w) =: l \cdot (u \otimes w).
\]
By linearity, this extends to all of $\mathcal{U}^{Z_1} \otimes \mathcal{W}^{Z_2}$. Since the map $\Psi$ is surjective, verifying $h \cdot v = v$ for all $h \in H$ is equivalent to verifying $l \cdot v = v$ for all $l \in L$, i.e. a vector in $\mathcal{U}^{Z_1} \otimes \mathcal{W}^{Z_2}$ is fixed by $H$ if and only if it is fixed by $L$.
\end{proof}
\vs
The isomorphism established in Theorem \ref{thm:amalgamated_fp} allows us to compute the dimension of the fixed-point space using the character theory of the finite group $L$.
\begin{corollary}[Dimension Parity Condition] \label{cor:dim_parity}Let $\operatorname{Irr}(L) = \{\mathcal L_i\}$ denote the set of irreducible $L$-representations. If the kernel-fixed-point spaces $\mathcal{U}^{Z_1}$ and $\mathcal{W}^{Z_2}$ admit $L$-isotypic decompositions of the form
\[
\mathcal{U}^{Z_1} \simeq \bigoplus_{i \in \mathcal I} a_i \mathcal L_i \quad \text{and} \quad \mathcal{W}^{Z_2} \simeq \bigoplus_{j \in \mathcal I} b_j \mathcal L_j,
\]
then the dimension of the $H$-fixed-point space in $\mathcal U \otimes \mathcal W$ is given by:
\begin{equation} \label{eq:dim_formula}
\dim (\mathcal{U} \otimes \mathcal{W})^H = \sum_{i \in \mathcal I} a_i b_i d_i, \quad d_i := \dim \operatorname{End}_L(\mathcal{L}_i).
\end{equation}
(Note: For the standard representations of geometric groups, such as $D_N$, $\mathbb Z_N$ or $O(2)$, we almost always have $d_i = 1$).
\end{corollary}
\begin{proof}
For any two real $L$-representations, there is a canonical isomorphism between the space of invariants in their tensor product and the space of equivariant linear maps (homomorphisms) between the factors (see \cite{tomDieck_rep} p. 26, Note 1.8.4), i.e.
\[
(\mathcal{U}^{Z_1} \otimes \mathcal{W}^{Z_2})^L \simeq \operatorname{Hom}_L((\mathcal{U}^{Z_1})^*, \mathcal{W}^{Z_2}).
\]
Since the representations are orthogonal, $(\mathcal{U}^{Z_1})^* \simeq \mathcal{U}^{Z_1}$, one has
\[
\dim (\mathcal{U}^{Z_1} \otimes \mathcal{W}^{Z_2})^L = \dim \operatorname{Hom}_L(\mathcal{U}^{Z_1}, \mathcal{W}^{Z_2}).
\]
Substituting our $L$-isotypic decompositions and accounting for $L$-isotypic multiplicities, yields
\begin{align*}
 \operatorname{Hom}_L\left( \bigoplus_{i \in \mathcal I} a_i \mathcal{L}_i, \bigoplus_{j \in \mathcal I} b_j \mathcal{L}_j \right) &\simeq \bigoplus_{i \in \mathcal I} \bigoplus_{j \in \mathcal I} \operatorname{Hom}_L(a_i \mathcal{L}_i, b_j \mathcal{L}_j) \\
 &\simeq \bigoplus_{i, j} a_i b_j \operatorname{Hom}_L(\mathcal{L}_i, \mathcal{L}_j).
\end{align*}
In the case of $i \neq j$, $\mathcal{L}_i$ and $\mathcal{L}_j$ are non-equivalent irreducible $L$-representations, so by Schur's lemma, we are guaranteed $\dim \operatorname{Hom}_L(\mathcal{L}_i, \mathcal{L}_j) = 0$. On the other hand, when $i=j$, the maps from an irreducible representation to itself form the endomorphism ring, such that $\dim \operatorname{Hom}_L(\mathcal{L}_i, \mathcal{L}_i) = d_i$. Taking dimensions on both sides, we are left only with the diagonal elements and the relation \eqref{eq:dim_formula}.
\end{proof}

\begin{corollary}[Amalgamated Parity Condition] \label{cor:amalgamated_parity}
Under the hypotheses of Corollary~\ref{cor:dim_parity}, the dimension of the $H$-fixed-point space is odd,
\[
\dim (\mathcal{U} \otimes \mathcal{W})^H \equiv 1 \pmod 2,
\]
if and only if there exist an \textbf{odd} number of $L$-isotypic indices $i \in \mathcal{I}$ satisfying the following three conditions simultaneously:
\begin{enumerate}
    \item[(i)] $a_i \equiv 1 \pmod 2$;
    \item[(ii)] $b_i \equiv 1 \pmod 2$;
    \item[(iii)] $d_i = 1$ (i.e., $\mathcal{L}_i$ is absolutely irreducible over $\mathbb{R}$).
\end{enumerate}
\end{corollary}
\begin{proof}
By Equation~\ref{eq:dim_formula}, the total dimension is the sum $\sum_i a_i b_i d_i$. A sum of integers is odd if and only if an odd number of its terms are odd. Furthermore, the product $a_i b_i d_i$ is odd if and only if all three of its factors are odd. By Frobenius' theorem, the endomorphism dimension $d_i \in \{1, 2, 4\}$, meaning $d_i$ is odd if and only if $d_i = 1$. Therefore, a term in the sum is odd exactly when $a_i$, $b_i$, and $d_i$ satisfy the three stated conditions.
\end{proof}
\section{The $G$-Equivariant Leray-Schauder Degree} \label{sec:appendix}
Given an isometric Banach $G$-representation $\mathscr H$ of functions taking values in an orthogonal $\Gamma$-representation $V$, the $G$-equivariant Leray-Schauder degree is a topological tool used to study the solution sets associated with equations of the form
\begin{align} \label{eq:1}
    \mathscr F(u) = 0, \quad u \in \mathscr H,
\end{align}
where $\mathscr F: \mathscr H \rightarrow \mathscr H$ is an operator satisfying the following conditions:
\begin{enumerate}[label=($B_\arabic*$)]
\item\label{c1} $\mathscr F$ is a $G$-equivariant completely continuous field;
\item\label{c2} there exists a sufficiently large $R > 0$ such that $\mathscr F$ is $B_R(0)$-admissibly $G$-homotopic to the identity operator $\operatorname{Id}: \mathscr H \rightarrow \mathscr H$;
\item\label{c3} $D \mathscr F(0): \mathscr H \rightarrow \mathscr H$ exists and the operator $\operatorname{Id} - D \mathscr F(0): \mathscr H \rightarrow \mathscr H$ is a $G$-equivariant compact field;
\item\label{c4} if $D \mathscr F(0): \mathscr H \rightarrow \mathscr H$ is an isomorphism, there exists a sufficiently small $\epsilon > 0$
such that $\mathscr F$ is $B_\epsilon(0)$-admissibly $G$-homotopic to $D \mathscr F(0)$.
\end{enumerate}
Condition \ref{c1} allows the problem of the existence of solutions for equation \eqref{eq:1} to be reformulated as a question concerning the non-triviality of the degree $\gdeg(\mathscr F, B(\mathscr H))$. In turn, Conditions \ref{c2}-\ref{c4} reduce calculation of $\gdeg(\mathscr F, B(\mathscr H))$ to a Burnside ring product involving a finite number of computationally simpler $G$-basic degrees.
\vs
\noi{\bf  Equivariant notation.}
Let $G$ be a compact Lie group. For any subgroup  $H \leq G$, we denote by $(H)$ its conjugacy class,
by $N(H)$ its normalizer by $W(H):=N(H)/H$ its Weyl group in $G$. The set of all subgroup conjugacy classes in $G$ is denoted by $\Phi(G):=\{(H): H\le G\}$ and has a natural partial order defined as follows
\[
(H)\leq (K) \iff \exists_{ g\in G}\;\;gHg^{-1}\leq K.
\]
As is possible with any partially ordered set, we extend the natural order over $\Phi(G)$ to a total order, which we indicate by $<$ to differentiate the two relations. Moreover, we put $\Phi_0 (G):= \{ (H) \in \Phi(G) : \text{$W(H)$  is finite}\}$ and, for any $(H),(K) \in \Phi_0(G)$, we denote by $n(H,K)$ the number of subgroups $\tilde K \leq G$ satisfying both $\wt K \in (K)$ and $H \leq \wt K$.
When considering a $G$-representation $\mathbb{V}$, we say that $(H)$ is \emph{maximal} in $\Phi_0(G; \mathbb{V})$ if any orbit type $(K) > (H)$ in $\mathbb{V}$ satisfies $(K) = (G)$.
\vs
Given a $G$-space $X$ with an element $x \in X$, we denote by
$G_{x} :=\{g\in G:gx=x\}$ the {\it isotropy group} of $x$
and we call $(G_{x}) \in \Phi(G)$  the {\it orbit type} of $x \in X$. Put $\Phi(G,X) := \{(H) \in \Phi_0(G)  : 
(H) = (G_x) \; \text{for some $x \in X$}\}$ and  $\Phi_0(G,X):= \Phi(G,X) \cap \Phi_0(G)$. For a subgroup $H\leq G$, the subspace $
X^{H} :=\{x\in X:G_{x}\geq H\}$ is called the {\it $H$-fixed-point subspace} of $X$. If $Y$ is another $G$-space, then a continuous map $f : X \to Y$ is said to be {\it $G$-equivariant} if $f(gx) = gf(x)$ for each $x \in X$ and $g \in G$.
\vs
\noi{\bf The Burnside ring.}
The free $\mathbb{Z}$-module $A(G) := \mathbb{Z}[\Phi_0(G)]$ has a natural ring structure when equipped with the multiplicative operation defined, for any pair of generators $(H),(K) \in \Phi_0(G)$, as follows
\begin{align*} 
    (H) \cdot (K) := \sum\limits_{(L) \in \Phi_0(G)} n_L(L), 
\end{align*}
where the coefficients $n_L \in \mathbb{Z}$ are given by the recurrence formula
\begin{align*} 
    n_L := \frac{n(L,H) |W(H)| n(L,K) |W(K)| - \sum_{(\tilde L) > (L)} n_{\tilde L} n(L,\tilde L) |W(\tilde L)|}{|W(L)|}.
\end{align*}
Any {\it Burnside ring} element $a \in A(G)$ can be expressed as a formal sum over some finite number of generator elements  
\[
a = n_1(H_1) + n_2(H_2) + \cdots + n_N(H_N),
\]
and we use the notation
\[
\operatorname{coeff}^H(a) := n_H,
\]
to specify the integer coefficient standing next to the generator element $(H) \in \Phi_0(G)$.
\vs
\noi{\bf An Axiomatic Construction of the $G$-equivariant Leray-Schauder degree.}
Let $\mathscr E$ be any isometric Banach $G$-representation. A map $f: \mathscr E \rightarrow \mathscr E$ is said to be a completely continuous $G$-equivariant field if it can be expressed in the form
\[
f(x) = x - F(x),
\]
for some compact $G$-equivariant map $F: \mathscr E \rightarrow \mathscr E$. Moreover, given an open bounded $G$-invariant set $\Om \subset \mathscr E$, a completely continuous $G$-equivariant field is said to be \emph{$\Om$-admissible} and the pair $(f,\Om)$ is called an \emph{admissible $G$-pair} if $f(x) \neq 0$ for all $x \in \partial \Om$. We denote by $\mathscr M^G(\mathscr E)$ the set of all admissible $G$-pairs in $\mathscr E$ and by $\mathscr M^G$ the set of all admissible $G$-pairs defined by taking a union over all isometric Banach $G$-representations as follows
\[
\mathscr M^G := \bigcup_{\mathscr E} \mathscr M^G(\mathscr E).
\]
The $G$-equivariant Leray-Schauder degree is defined as the unique map $\gdeg: \mathscr M^G \rightarrow A(G)$ that assigns to every admissible $G$-pair $(f,\Omega)$ the Burnside ring element
	\begin{align}
		\label{eq:G-deg0}\gdeg(f,\Omega)=\sum_{(H) \in \Phi_0(G)}%
		{n_{H}(H)},
	\end{align}
satisfying the three degree axioms
	\begin{itemize}
		\item[] \textbf{(Additivity)} 
  For any two  disjoint open $G$-invariant subsets
  $\Omega_{1}$ and $\Omega_{2}$ with
		$f^{-1}(0)\cap\Omega\subset\Omega_{1}\cup\Omega_{2}$, one has
		\begin{align*}
			\gdeg(f,\Omega)=\gdeg(f,\Omega_{1})+\gdeg
			(f,\Omega_{2}).
		\end{align*}

		\item[] \textbf{(Homotopy)} For any 
  $\Omega$-admissible $G$-homotopy, $h:[0,1]\times V\to V$, one has
		\begin{align*}
			\gdeg(h_{t},\Omega)=\mathrm{constant}.
		\end{align*}

		\item[] \textbf{(Normalization)}
  For any open bounded neighborhood of the origin in an isometric Banach $G$-representation $\mathscr E$ with the identity operator $\id:\mathscr E \rightarrow \mathscr E$, one has
		\begin{align*}
			\gdeg(\id,\Omega)=(G).
		\end{align*}
	\end{itemize}
The following are three additional properties of the map $\gdeg$ which can be derived from the axiomatic properties defined above (cf. \cite{book-new, AED}):		
\begin{itemize}
		\item[] \textbf{(Existence)} If  $n_{H} \neq0$ for some $(H) \in \Phi_0(G)$ in \eqref{eq:G-deg0}, then there
		exists $x\in\Omega$ such that $f(x)=0$ and $(G_{x})\geq(H)$.
		\item[] {\textbf{(Multiplicativity)}} For any $(f_{1},\Omega
		_{1}),(f_{2},\Omega_{2})\in\mathcal{M} ^{G}$,
		\begin{align*}
			\gdeg(f_{1}\times f_{2},\Omega_{1}\times\Omega_{2})=
			\gdeg(f_{1},\Omega_{1})\cdot \gdeg(f_{2},\Omega_{2}),
		\end{align*}
		where the multiplication `$\cdot$' is taken in the Burnside ring $A(G )$.

		\item[] \textbf{(Recurrence Formula)} For an admissible $G$-pair
		$(f,\Omega) \in \mathscr M^G(\mathscr E)$, the $G$-degree \eqref{eq:G-deg0} can be computed using the
		following Recurrence Formula:
		\begin{equation}
			\label{eq:RF-0}n_{H}=\frac{\deg(f^{H},\Omega^{H})- \sum_{(K)>(H)}{n_{K}\,
					n(H,K)\, \left|  W(K)\right|  }}{\left|  W(H)\right|  },
		\end{equation}
		where $\left|  X\right|  $ stands for the number of elements in the set $X$
		and $\deg(f^{H},\Omega^{H})$ is the Brouwer degree of the map $f^{H}%
		:=f|_{\mathscr E^{H}}$ on the set $\Omega^{H}\subset \mathscr E^{H}$.
	\end{itemize}
\noi{\bf Computation of the $G$-Equivariant Degree for a $G$-Equivariant Linear Isomorphism.} 
Given an orthogonal $G$-representation $V$ with a $G$-isotypic decomposition
\[
V \simeq \bigoplus_{i \in \mathcal I} V_i, \quad V_i \simeq m_i \mathcal V_i,
\]
and any bounded $G$-equivariant linear isomorphism $T:V\to V$, the computation of the $G$-equivariant Leray-Schauder degree $\gdeg(T, B(V)) \in A(G)$ generally requires evaluating the Recurrence Formula \eqref{eq:RF-0} sequentially down the isotropy lattice. However, when evaluating the coefficient of a maximal orbit type $(H) \in \Phi_0(G; V)$, the recurrence relation collapses, allowing the coefficient to be determined strictly by the parity of the negative spectrum of $T$. 
\vs
The following Lemma establishes a direct algebraic criterion for the non-trivial contribution of a maximal orbit type to the equivariant degree:
\begin{lemma}\label{lemm:nontriviality_conditions}
Let $(H) \in \Phi_0(G; V)$ be a \textbf{maximal} orbit type. The coefficient 
\[
n_H = \operatorname{coeff}^H \big(\gdeg(T, B(V))\big),
\]
is non-zero if and only if, for an \textbf{odd} number of non-trivial $G$-isotypic indices $i \in \mathcal{I} \setminus \{0\}$, the following two conditions are simultaneously satisfied:
\begin{enumerate}
    \item[(i)] The negative eigenspace of the restriction $T_i := T|_{V_i}: V_i \to V_i$ consists of an \textbf{odd} number of copies of the irreducible $G$-representation $\mathcal V_i$;
    \item[(ii)] the $H$-fixed point space in the irreducible representation $\mathcal V_i$ is \textbf{odd}-dimensional.
\end{enumerate}
\end{lemma}
\begin{proof}
By the Recurrence Formula \eqref{eq:RF-0}, the coefficient $n_H$ is given by:
\[
n_H = \frac{\deg(T^H, B(V^H)) - \sum_{(K) > (H)} n_K n(H,K) |W(K)|}{|W(H)|}.
\]
Since $(H)$ is a maximal orbit type in $\Phi_0(G; V)$, the fixed-point space satisfies $V^K = V^G$ for every subgroup $K$ strictly containing $H$. On the space $V^K$, the linear map restricts to $T^G$, yielding the Brouwer degree $\deg(T^K, B(V^K)) = \deg(T^G, B(V^G))$ for all $K > H$. By downward induction, this yields $n_G = \deg(T^G, B(V^G))$ and $n_K = 0$ for all $H < K < G$, simplifying the recurrence relation to:
\[
n_H = \frac{\deg(T^H, B(V^H)) - \deg(T^G, B(V^G))}{|W(H)|}.
\]
The local Brouwer degree of a non-singular linear operator on the unit ball is given by $\deg(T^H, B(V^H)) = (-1)^{\mu_-^H(T)}$, where $\mu_-^H(T)$ denotes the sum of the algebraic multiplicities of the strictly negative real eigenvalues of $T^H: V^H \to V^H$ {\red (since any complex eigenvalues occur in conjugate pairs, their product is strictly positive; consequently, they preserve the sign of the determinant and are excluded from the parity count)}. Similarly, $\deg(T^G, B(V^G)) = (-1)^{\mu_-^G(T)}$.
Thus, $n_H \neq 0$ if and only if $\mu_-^H(T)$ and $\mu_-^G(T)$ have opposite parities, i.e., $\mu_-^H(T) - \mu_-^G(T) \equiv 1 \pmod 2$.
We can compute this difference by summing the negative eigenvalue contributions from each non-trivial $G$-isotypic component $V_i$. Since $V^G$ consists entirely of trivial representations, $\mu_-^H(T) - \mu_-^G(T) = \sum_{i \neq 0} c_i \dim \mathcal{V}_i^H$, where $c_i$ is the number of irreducible copies of $\mathcal{V}_i$ in the negative eigenspace of $T_i$.
This sum is an odd integer if and only if an odd number of products $c_i \dim \mathcal{V}_i^H$ are odd.
A product is odd iff both factors are odd, which corresponds exactly to the simultaneous satisfaction of conditions (i) and (ii) for $i \neq 0$.
\end{proof}

\section{The Mod-2 $G$-Equivariant Spectral Flow} \label{sec:appendix_esf}
While the classical equivariant spectral flow (see \cite{Izydorek2021}) takes values in the real representation ring $RO(G)$ over the integers, the topological properties of the equivariant Leray--Schauder degree---specifically its reliance on the local Brouwer degree on fixed-point spaces---naturally restrict the relevant spectral information to a parity count. In this appendix, we formally define the spectral flow representation $\mathbb{V}_{[a,b]}$ as an element of a mod-2 representation module.

\vs
\noi{\bf The Mod-2 Representation Ring.}
Let $G$ be a compact Lie group and let $\operatorname{Irr}(G) = \{\mathcal{V}_i\}_{i \in \mathcal{I}}$ be a complete list of the $G$-isomorphism classes of irreducible $G$-representations. The standard real representation ring $RO(G)$ is the free $\mathbb{Z}$-module generated by $\operatorname{Irr}(G)$. To capture the parity of eigenvalue crossings, we define the \emph{mod-2 representation ring} as the free $\mathbb{Z}_2$-module obtained via the tensor product
\[
RO_2(G) := RO(G) \otimes_{\mathbb{Z}} \mathbb{Z}_2.
\]
Any element $[\mathbb{W}] \in RO_2(G)$ can be expressed as a formal sum over the irreducible generators
\[
[\mathbb{W}] = \sum_{i \in \mathcal{I}} \rho_i [\mathcal{V}_i], \quad \rho_i \in \mathbb{Z}_2.
\]
Equipped with the standard direct sum and tensor product of representations (evaluated modulo 2), $RO_2(G)$ inherently tracks the parity of $G$-isotypic multiplicities.

\vs
\noi{\bf The Mod-2 Spectral Index.}
Let $\mathscr{H}$ be an isometric Banach $G$-representation and $T: \mathscr{H} \to \mathscr{H}$ an equivariant, completely continuous linear isomorphism. The negative spectrum of $T$ is a finite-dimensional, $G$-invariant subspace $E^-(T) \subset \mathscr{H}$. We define the \emph{mod-2 spectral index} of $T$ as the projection of this negative eigenspace into the mod-2 representation ring:
\[
\mu_G(T) := [E^-(T)] \pmod{2} \;\in RO_2(G).
\]
Explicitly, if $E^-(T)$ admits the $G$-isotypic decomposition $E^-(T) \simeq \bigoplus_{i \in \mathcal{I}} c_i \mathcal{V}_i$, the mod-2 spectral index simplifies the multiplicities to $c_i \pmod{2}$.

\vs
\noi{\bf The Mod-2 Equivariant Spectral Flow.}
Consider the parametrized family of linearizations $\mathscr{A}(\alpha): \mathscr{H} \to \mathscr{H}$ defined in Section \ref{sec:bifurcation_setup}. Let $a, b \in \mathbb{R}$ be any two regular parameter values. The net topological change across the interval $[a,b]$ is governed by the eigenvalues crossing the origin. We define the \emph{mod-2 equivariant spectral flow} of $\mathscr{A}(\alpha)$ across $[a,b]$ as the difference of the boundary spectral indices:
\[
\mathbb{V}_{[a,b]} := \mu_G(\mathscr{A}(a)) - \mu_G(\mathscr{A}(b)) \;\in RO_2(G).
\]
Since the coefficients live in $\mathbb{Z}_2$, this difference is equivalent to the symmetric difference of the negative eigenspaces, explicitly yielding
\[
\mathbb{V}_{[a,b]} = \sum_{i \in \mathcal{I}} \rho_i [\mathcal{V}_i], \quad \rho_i \equiv c_i(a) - c_i(b) \pmod{2},
\]
which precisely recovers the operational definition of the spectral flow representation introduced in \eqref{def:spectral_flow_representation}. By embedding $\mathbb{V}_{[a,b]}$ within $RO_2(G)$, the Ize pair parity condition (Theorem \ref{thm:parity_guarantee}) can be understood algebraically as an evaluation of fixed-point dimensions on the support of this mod-2 invariant.

\end{document}